\documentclass{amsart}
\usepackage{amssymb,amsmath,amscd,xy,graphicx,textcomp}
\newtheorem{theorem}{Theorem}[section]
\newtheorem{lemma}[theorem]{Lemma}
\newtheorem{corollary}[theorem]{Corollary}
\newtheorem{definition}[theorem]{Definition}

\newtheorem{remark}[theorem]{\it Remark}
\newtheorem{example}[theorem]{Example}
\newtheorem{proposition}[theorem]{Proposition}

\xyoption{arrow}

\xyoption{matrix}

\def\C{\mathbb{C}}
\def\R{\mathbb{R}}
\def\Z{\mathbb{Z}}

\def\tree{\mathcal{T}}
\def\trop{\mathbb{T}}

\def\ql{\backslash \! \backslash}

\title{Toric degenerations and tropical geometry of branching algebras}
\author{Christopher Manon} 
\thanks{University of California, Berkeley, chris.manon@math.berkeley.edu, This work was supported by the NSF fellowship DMS-0902710}

\begin{document}

\begin{abstract}
We construct polyhedral families of valuations on the branching algebra of a morphism of reductive groups. 
This establishes a connection between the combinatorial rules for studying a branching problem and the tropical geometry of the branching algebra.  
In the special case when the branching problem comes from the inclusion of a Levi subgroup or a diagonal subgroup, we 
use the dual canonical basis of Lusztig and Kashiwara to build toric deformations of the branching algebra.  
\end{abstract}










\maketitle

\tableofcontents


\section{Introduction}

For every map $\phi:H \to G$ of connected reductive groups over $\C,$ we may regard the representations
of $G$ as representations of $H$ via its action through $\phi.$  For any irreducible representation $V(\lambda)$ 
of $G$ there is a direct sum decomposition into irreducible $H$ representations. 

\begin{equation}
V(\lambda) = \bigoplus_{\beta \in \Delta_H} Hom_H(V(\beta), V(\lambda)) \otimes V(\beta)\\
\end{equation}

\noindent 
The vector spaces $Hom_H(V(\beta), V(\lambda))$ determine the restriction functor $\phi^*: Rep(G) \to Rep(H)$
on the categories of finite dimensional representations of $H$ and $G,$ and a calculation of the dimensions of these spaces is called a branching rule associated to $\phi.$ Recall that equivalence classes of irreducible representations of $G$ are in bijection with the lattice points of convex rational cone $\Delta_G,$ a Weyl chamber of $G.$   Solving the branching problem can be seen as assigning non-negative integers, the dimensions of
the above spaces, to the dominant weights in $\Delta_H \times \Delta_G.$ This leads naturally to questions about the nature of the sets of points
where these numbers are non-zero.  One general approach to this problem is to bring in techniques from commutative algebra by integrating the spaces $Hom_H(V(\beta), V(\lambda))$ into a commutative algebra, $R(\phi).$  This algebra is constructed from
the algebra of invariants by a maximal unipotent $U^+ \subset G$ in the coordinate ring $\C[G].$  As a representation of $G,$ this algebra is the multiplicity free sum of the irreducible representations of $G.$

\begin{equation}
R_G = \C[G]^U = \bigoplus_{\lambda \in \Delta_G} V(\lambda)\\
\end{equation}

\noindent
This algebra is finitely generated by the representations associated
to the fundamental weights in $\Delta_G.$  Multiplication can be computed 
component-wise by the Cartan product, which is projection onto the highest weight component of the tensor product. 

\begin{equation}
V(\lambda) \otimes V(\eta) \to V(\lambda + \eta)\\
\end{equation}

\begin{definition}
For $\phi:H \to G$ a map of reductive groups, the branching algebra for $\phi$ is defined as the invariant ring of $R_H \otimes R_G$ by $H,$ or equivalently as the $U_H-$invariants of $R_G,$ where $U_H \subset H$ is a maximal unipotent subgroup.  

\begin{equation}
R(\phi) = [R_H \otimes R_G]^H =  R_G^{U_H}\\
\end{equation}

\end{definition}

It is worth mentioning two special cases of this construction.  For $\phi = \bold{id},$ the algebra $R(\bold{id})$ is the subalgebra of highest weight vectors in $R_G,$ in other words it is the toric algebra attached to the semigroup of dominant weights of $G,$ $\C[\Delta].$    For the unity morphism $i_G: 1 \to G,$ the algebra is just $R_G.$ For all $\phi,$ $R(\phi)$ is finitely generated and graded by the monoid of dominant weights $\Delta_H \times \Delta_G,$ and the non-zero branching numbers for $\phi:H \to G$ correspond to the non-zero graded components of $R(\phi),$ we call this submonoid $C(\phi) \subset \Delta_H \times \Delta_G.$  Since $R(\phi)$ is finitely generated, the monoid $C(\phi)$ is finitely generated as well, this implies that its non-negative real saturation is a convex rational cone in $\Delta_H\times \Delta_G.$

The cone $C(\phi)$ itself can be difficult to understand so in order to get more information about the polyhedral and algebraic properties of $C(\phi)$ it helps to find a cone $D$ which can be understood more easily, and a surjection $D \to C(\phi).$  Conceptually, the cone $D$ "enriches" the branching structure described by $C(\phi),$ and meaningful $D$ correspond to combinatorial rules which solve or simplify the branching problem, such as the Littlewood-Richardson rule or the Pieri rule in type $A.$  Throughout this paper we will describe two ways a cone like $D$ can be constructed, how to combine these constructions, and some of the resulting algebraic and geometric implications for $R(\phi)$.  The principle driving our constructions is that different meaningful $D,$ arising from different combinatorial rules for understanding branching problems, correspond to flat degenerations of $R(\phi),$ and that these degenerations fit together into a polyhedral structure that is described by elements of tropical geometry.

The connection with tropical geometry comes from our use of valuations to construct the degenerations of $R(\phi).$   It is a theorem from tropical geometry (see \cite{P} and Section \ref{trop}) that the tropical variety $tr(I)$ of any ideal $I$ from a presentation of $R(\phi)$ can be constructed as an image of $\mathbb{V}_{\trop}(R(\phi)),$ the set of all valuations on $R(\phi)$ into the tropical real line $\trop$ which are trivial on $\C \subset R(\phi).$   We define a complex of rational, pointed cones $K_{\phi}$ for each morphism $\phi$ in the category of reductive groups which is informed by the factorizations of $\phi,$ and we show that points in this complex define valuations on the branching algebra $R(\phi).$   These valuations do not always have toric associated graded algebras, so we explain how they can be enhanced to toric degenerations of $R(\phi)$ in the special cases when $\phi$ is a diagonal embedding $\delta_n: G \to G^n$ or an inclusion of a Levi subgroup $i_L:L \to G$.

\subsection{Construction of the branching complex}

For any branching problem $\phi: H \to G$ it is possible to enrich $C(\phi)$ by studying factorizations of $\phi.$ 

\begin{equation}
\phi = \phi_m \circ \ldots \circ \phi_0\\
\end{equation}
$$
\begin{CD}
H @>\phi_0>> K_1 @>\phi_1>> \ldots @>\phi_{m-1}>> K_m @>\phi_m>> G\\
\end{CD}
$$

\noindent
By the semisimplicity of the categories $Rep(K_i),$ the space $Hom_H(V(\eta), V(\lambda))$ decomposes into a direct sum. 

\begin{equation}
Hom_H(V(\eta), V(\lambda)) = \bigoplus_{\eta, \vec{\tau}, \lambda \in \Delta_{H \times \ldots \times G}} Hom_H(V(\eta), V(\tau_1))\otimes \ldots \otimes Hom_{K_m}(V(\tau_m), V(\lambda))\\
\end{equation}

\noindent
We think of elements of these spaces as diagrams of intertwiners .

$$
\begin{CD}
V(\eta) @>f_0>> V(\tau_1) \ldots V(\tau_{m}) @>f_m>> V(\lambda)\\
\end{CD}
$$

\noindent
We abbreviate the summands on the right above as $W(\vec{\phi}, \vec{\tau}) \subset R(\phi),$ and $R(\phi)$ has a direct sum decomposition into these spaces as a vector space. 
Forgetting all but the dominant weight data results in a fiber product cone $C(\vec{\phi}) = C(\phi_0) \times_{\Delta_{K_1}} \ldots \times_{\Delta_{K_m}} C(\phi_m)$ with a natural map to  $C(\phi_m \circ\ldots \circ \phi_0) = C(\phi).$  On the level of commutative algebra the above direct sum decomposition is a filtration of the multiplication operation on $R(\phi)$. The following will be discussed in Section \ref{branch}.

\begin{theorem}\label{branchfiltration}
For any factorization of a morphism of reductive groups $\phi = \phi_n \circ \ldots \circ \phi_1,$ there is a filtration of $R(\phi)$ defined on the spaces $W(\vec{\phi}, \vec{\lambda}),$ by the ordering on the dominant weight labels.

\begin{equation}
W(\vec{\phi}, \vec{\lambda}) W(\vec{\phi}, \vec{\eta}) \subset \bigoplus_{\vec{\beta} \leq \vec{\lambda} + \vec{\eta}} W(\vec{\phi}, \vec{\beta})\\
\end{equation}

\noindent
The associated graded algebra is isomorphic to $[\otimes_{i =1}^n R(\phi_i)]^{T_1 \times \ldots \times T_{n-1}}  \subset  \otimes_{i =1}^n R(\phi_i).$  Furthermore,  $[\otimes_{i =1}^n R(\phi_i)]^{T_1 \times \ldots \times T_{n-1}}$ is a flat degeneration of $R(\phi).$
\end{theorem}

\noindent
The torus $T_1 \times \ldots \times T_{n-1}$ is a product of the maximal tori in the 
groups $K_i.$  A general summand of  $\otimes_{i =1}^n R(\phi_i)$ is of the form

\begin{equation}
Hom_{H}(V(\lambda_0), V(\eta_0)) \otimes Hom_{K_1}(V(\lambda_1), V(\eta_1))\otimes \ldots \otimes Hom_{K_{n-1}}(V(\lambda_{n-1}), V(\eta_{n-1})),\\
\end{equation}

\noindent
and this algebra has an action of $T_1^2 \times \ldots \times T_{n-1}^2,$  where $T_i^2$ acts with weight $(\eta_{i-1}, \lambda_i).$
We take invariants by the action by the anti-diagonal subtorus $T_1 \times \ldots \times T_{n-1} \subset T_1^2 \times \ldots \times T_{n-1}^2,$ 
so the components in the invariant algebra satisfy $\eta_{i-1} = \lambda_i.$

To each such factorization we will then associate a rational polyhedral cone $B(\vec{\phi}).$ 
We take the product $\Delta_H^* \times \ldots \times \Delta_G^*$ of the dual Weyl-chambers of the groups, in other words the Weyl chambers of their respective Langlands dual groups.   These elements can be realized as linear functionals on the Weyl chambers $\Delta_H \times \ldots \times \Delta_G,$ so we define $B(\vec{\phi})$ as the quotient of $\Delta_H^* \times \ldots \times \Delta_G^*$ by the relation which identifies two elements having the same values on $C(\vec{\phi}).$  The following will be proved in Section \ref{branch}.

\begin{proposition}\label{valcon}
There is a map $f_{\vec{\phi}}: B(\vec{\phi}) \to \mathbb{V}_{\trop}(R(\phi)).$
\end{proposition}

We construct a complex out of the cones $B(\vec{\phi})$
as $\vec{\phi}$ runs over all factorizations of $\phi$ in the category of reductive groups.

\begin{definition}
Let $\mathcal{K}_{\phi}$ be the simplicial set defined as follows.  The $k$-simplices are defined as length-$k$ factorizations $\phi_k \circ \ldots \circ \phi_1 = \phi$ of the morphism $\phi.$  The face and degeneracy maps are defined by the following operations on chains of morphisms. 

\begin{equation}
d_k^i(\phi_1, \ldots, \phi_i, \phi_{i+1}, \ldots, \phi_k) = (\phi_1, \ldots, \phi_{i+1} \circ \phi_i, \ldots, \phi_k)\\
\end{equation}

\begin{equation}
s_k^i(\phi_1, \ldots, \phi_i, \phi_{i+1}, \ldots, \phi_k) = (\phi_1, \ldots, \phi_i, id, \phi_{i+1}, \ldots \phi_k)\\
\end{equation}
\end{definition}

To make this well-defined, we can take some small category in the category of reductive groups which contains all such factorizations up to equivalence. The cones $B(\vec{\phi})$ carry maps which correspond to the degeneracy maps in the above simplicial set by setting the appropriate coordinate to be $0,$ and the face maps are realized by adding the two adjacent coordinates which come from the same group.

\begin{equation}
(d_k^i)^*:B(\phi_1, \ldots, \phi_i \circ \phi_{i-1}, \ldots \phi_k) \to B(\phi_1, \ldots, \phi_k)\\ 
\end{equation}

\begin{equation}
(d_k^i)^*: (\rho_0, \ldots, \rho_k) \to (\rho_0, \ldots, 0, \ldots, \rho_k)\\
\end{equation}

\begin{equation}
(s_k^i)^*: B(\phi_1, \ldots, \phi_i, id, \phi_{i+1}, \ldots, \phi_k) \to  B(\phi_1, \ldots, \phi_i, \phi_{i+1}, \ldots, \phi_k)\\ 
\end{equation}

\begin{equation}
(s_k^i)^*:(\rho_0, \ldots,\rho_i, \rho_{i+1},\ldots, \rho_k) \to (\rho_0, \ldots, \rho_i + \rho_{i+1}, \ldots, \rho_k)\\
\end{equation}

\noindent
We can glue the $B(\vec{\phi})$ together along these maps by taking a colimit to obtain a polyhedral model of $\mathcal{K}_{\phi},$ which we call $K_{\phi}.$  These gluing maps will 
be shown to be compatible with the maps $f_{\vec{\phi}}.$

\begin{theorem}\label{valcom}
There is a map $F_{\phi}: K_{\phi} \to \mathbb{V}_{\trop}(R(\phi)).$
\end{theorem}

In particular, every tropical variety $tr(I)$ associated to a presentation of the ring $R(\phi)$ has a subset which is the image of the polyhedral complex $K_{\phi}.$  In this way, the structure of the category of reductive groups informs the tropical geometry of branching algebras.

We now describe special factorizations of two types of morphisms, a diagonal $\delta_n: G \to G^n$ and the inclusion of a Levi subgroup $i_L: L \to G.$  The degenerations defined by these special factorizations can be completed to toric degenerations.   As a vector space, the branching algebra $R(\delta_n)$ is a direct sum of the invariants in all $n+1$-fold tensor products of representations of $G,$ for this reason we call it the full tensor algebra of $G.$
 
\begin{equation}
R(\delta_n) = \bigoplus_{\vec{\lambda} \in \Delta^{n+1}} Hom_G(V(\lambda_0), V(\lambda_1) \otimes \ldots \otimes V(\lambda_n)) = \bigoplus_{\vec{\lambda} \in \Delta^{n+1}} [V(\lambda_0^*) \otimes \ldots \otimes V(\lambda_n)]^G\\
\end{equation}

\noindent
We make a special application of the branching construction described above to reduce the study of $R(\delta_n)$ to $R(\delta_2)$ for any $n.$  

\begin{definition}
Define an oriented $n$-tree to be a $\tree$ with $n+1$ leaves labeled $0, \ldots, n,$ where
each edge is directed in such a way that each non-leaf vertex has exactly one in-edge, 
the $0$ leaf is a source, and all other leaves are sinks.     
\end{definition}

\noindent 
Note that the orientation is entirely defined by the labeling $\{0, \ldots, n\}.$  To each oriented $n$-tree we define a factorization of $\delta_n$ by assigning each vertex
$v \in V(\tree)$ the morphism $\delta_{val(v)-1}:G \to G^{val(v) -1},$ where $val(v)$ is the valence of $v$.

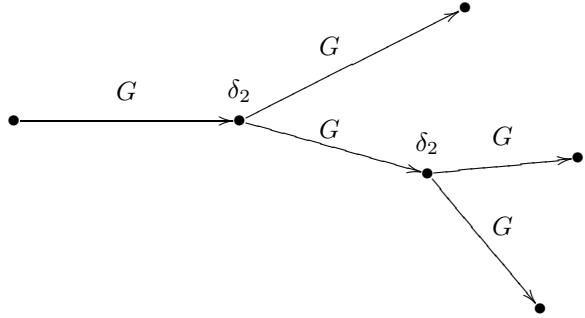
\begin{figure}[htbp]
\centering
\begin{xy}
(-15, 4)*{G};
(12, 10)*{G};
(35, -2)*{G};
(35, -14)*{G};
(12, -1)*{G};
(0,4)*{\delta_2};
(25, -3 )*{\delta_2};
(0,0)*{\bullet} = "AA";
(30,15)*{\bullet} = "BB";
(45,-5)*{\bullet} = "CC";
(40,-25)*{\bullet} = "DD";
(-30,0)*{\bullet} = "EE";
(25,-7)*{\bullet} = "FF";
"BB"; "AA";**\dir{-}? >* \dir{>};
"AA"; "EE";**\dir{-}? >* \dir{>};
"FF"; "AA";**\dir{-}? >* \dir{>};
"DD"; "FF";**\dir{-}? >* \dir{>};
"CC"; "FF";**\dir{-}? >* \dir{>};
\end{xy}
\caption{A factorization, $\delta_3 = [id \times \delta_2]\circ \delta_2.$}
\end{figure}

The branchings defined by this tree construction define a connected subcomplex of $K_{\delta_n},$ which we call $D_n(G).$
The factorization diagrams corresponding to $\tree$ are labellings of the edges $E(\tree)$ of $\tree$ by dominant $G$-weights, and the vertices $V(\tree)$ of $\tree$ by $G$-tensors. By Theorem \ref{branchfiltration} above, this results in a filtration on $R(\delta_n)$ with associated graded algebra a subalgebra of torus invariants in $\bigotimes_{v \in V(\tree)} R(\delta_{val(v)-1}).$   If $\tree$ is trivalent, we obtain a subalgebra of of $[R(\delta_2)]^{\otimes n-2}.$ 

\begin{remark}
The complex $D_n(SL_2(\C))$ can be recognized as the space of phylogenetic trees defined by Billera, Holmes and Vogtman in \cite{BHV} to give a geometric context to phylogenetic algorithms from mathematical biology.
\end{remark}

There is a distinguished class of factorizations we can use on the map defined by the inclusion of a Levi subgroup $i_L: L \to G$ as well, namely factorizations by other Levi subgroups.  For any chain of Levi subgroups 

$$
\begin{CD}
L @>i_{L, L_1}>> L_1 @>i_{L_1, L_2}>> \ldots @>i_{L{m-1}, L_m} >> L_m @>i_{L_m, G}>> G\\
\end{CD}
$$

\noindent
we obtain a cone of valuations $B(\vec{i}_L),$ with a flat degeneration of
$R(i_{L, G})$ to  $[R(i_{L, L_1}) \otimes \ldots \otimes R(i_{L_k, G})]^{T_1 \times \ldots \times T_m}.$  Each such cone can be represented combinatorially by a nesting of the Dynkin diagrams defining the $L_i.$  

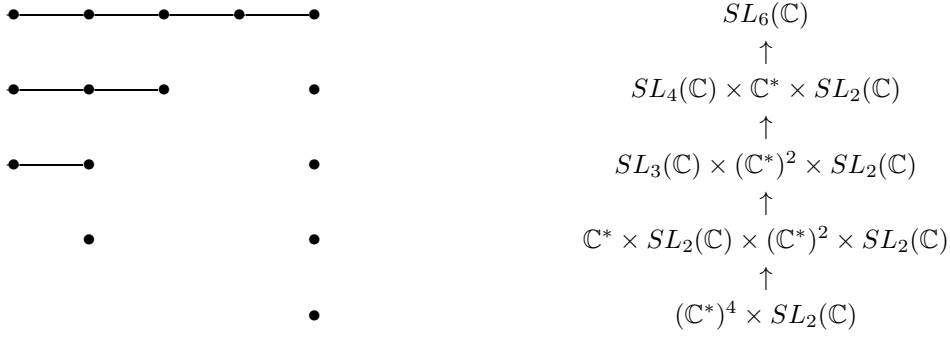
\begin{figure}[htbp]
\centering
\begin{xy}
(100, 0)*{SL_6(\C)};
(100, -10)*{SL_4(\C)\times \C^* \times SL_2(\C)};
(100, -20)*{SL_3(\C)\times (\C^*)^2 \times SL_2(\C)};
(100, -30)*{\C^*\times SL_2(\C) \times (\C^*)^2 \times SL_2(\C)};
(100, -40)*{(\C^*)^4\times SL_2(\C)};
(100, -5)*{\uparrow};
(100, -15)*{\uparrow};
(100, -25)*{\uparrow};
(100, -35)*{\uparrow};
(0,0)*{\bullet} = "A";
(10,0)*{\bullet} = "B";
(20,0)*{\bullet} = "C";
(30,0)*{\bullet} = "D";
(40,0)*{\bullet} = "E";
"A"; "B";**\dir{-}? >* \dir{-};
"B"; "C";**\dir{-}? >* \dir{-};
"C"; "D";**\dir{-}? >* \dir{-};
"D"; "E";**\dir{-}? >* \dir{-};
(0,-10)*{\bullet} = "A1";
(10,-10)*{\bullet} = "B1";
(20,-10)*{\bullet} = "C1";
(40,-10)*{\bullet} = "E1";
"A1"; "B1";**\dir{-}? >* \dir{-};
"B1"; "C1";**\dir{-}? >* \dir{-};
(0,-20)*{\bullet} = "A2";
(10,-20)*{\bullet} = "B2";
(40,-20)*{\bullet} = "E2";
"A2"; "B2";**\dir{-}? >* \dir{-};
(10,-30)*{\bullet} = "B3";
(40,-30)*{\bullet} = "E3";
(40,-40)*{\bullet} = "E4";
\end{xy}
\caption{A chain of sub-Dynkin diagrams of type A}
\end{figure}

\noindent
We define $H_L(G) \subset K_{i_{L,G}}$ to be the connected subcomplex defined by all chains of
Levi subgroups which begin with $L$ for a fixed choice of maximal torus.  Just as exploring $D_n(G)$ for $G = SL_m(\C)$ has lead to some interesting combinatorics (see \cite{M3}, \cite{SpSt}), we anticipate that the study of $H_L(G)$ for various $L$ and $G$ to yield interesting combinatorial structures as well.

\subsection{The dual canonical basis and branching algebras}

Factorization diagrams provide a convenient combinatorial filtration of general
branching algebras, but the associated graded algebras of these filtrations are not
always monoidal.  Now we describe how to use the dual canonical basis (crystal basis) of Lusztig to complete the branching degenerations associated to 
the valuations in $D_n(G)$ and $H_L(G)$  to toric degenerations.

 A full solution to a branching problem would be a cone $D \to C(\vec{\phi})$  where each element of a basis of branching maps $f: V(\beta) \to V(\lambda), (\beta, \lambda) \in \Delta_H \times \Delta_G$ is assigned a unique element of $D$.  Ideally, one wants

\begin{enumerate}
\item A finitely generated monoid $S(\phi)$\\
\item A map $\pi: S(\phi) \to \Delta_H \times \Delta_G$ which recovers $C(\phi)$, such that the number of points
in the fiber over $(\lambda, \eta)$ is $dim[Hom_H(V(\lambda), V(\eta))].$\\
\item An algebraic relationship between $S(\phi)$ and $R(\phi).$
\end{enumerate}

\noindent
We will use the convex polyhedral descriptions of branching for 
inclusion of Levi subgroup $i_{L,G}:L \subset G$ and the diagonal map
$\delta_2:G \to G^2$ given by Berenstein and Zelevinsky in \cite{BZ1}, \cite{BZ2}
to give a variety of solutions to $1$, $2$, and $3$ above in these cases.\

The dual canonical basis $B \subset R_G$ has a parametrization by $N-$tuples of non-negative integers for each  reduced decomposition $\bold{i} = (i_1, \ldots, i_N)$  of the longest word of the Weyl group $w_0 \in W$ of $G$.

\begin{equation}
w_0 = s_{\alpha_{i_1}}\ldots s_{\alpha_{i_N}}\\
\end{equation}

\noindent
This defines an injective map $B \to \Delta \times \Z_{\geq 0}^N,$
and a partial ordering on $B,$ where $b_{\lambda, \vec{s}} > b_{\eta, \vec{t}}$
if $\lambda > \eta$ as dominant weights, or $\lambda = \eta$ and $\vec{s} > \vec{t}$
lexicographically.  The image of this map is the set of integer points in a convex cone $C(\bold{i}) \subset \Delta \times \Z_{\geq 0}^N.$  Multiplication in $R_G$ with respect to the basis $B$ was shown to be lower-triangular by Caldero \cite{C},
see also \cite{K}. 

\begin{equation}
b_{\lambda, \vec{s}} \times b_{\eta, \vec{t}} = b_{\lambda + \eta, \vec{s} + \vec{t}} + \sum_{\vec{\ell} < \vec{s}} C_{\lambda + \eta, \vec{\ell}} b_{\lambda + \eta, \vec{\ell}}\\
\end{equation}

\noindent
This is similar to the filtration defined by factorization diagrams.   We will employ an observation we believe to be essentially due to Zhelobenko \cite{Zh} to realize
the branching algebras $R(\delta_2)$ and $R(i_{L,G})$ as subalgebras of $R_G\otimes \C[T]$ and $R_G$ respectively in order to show the following. 

\begin{theorem}\label{basisinherit}
The algebras $R(\delta_2)$ and $R(i_{L,G})$ inherit the dual canonical basis
along with its lower triangular multiplication property from $R_G.$
\end{theorem}

Using the work of Berenstein and Zelevinsky \cite{BZ3}, we observe that the
the bases of $R(\delta_2)$ and $R(i_{L,G})$ both come with natural labelings by the
integer points in convex rational cones $C_3(\bold{i})$ and $C_L(\bold{i}),$ respectively, one
for each choice of string parameter $\bold{i}.$  

\begin{corollary}
The algebras $R(\delta_2)$ and $R(i_{L,G})$ carry flat degenerations $\C[C_3(\bold{i})]$
and $\C[C_L(\bold{i})]$ respectively. 
\end{corollary}

In the examples we explain a particular nice instance of a cone $C_3(\bold{i})$ in type $A.$
These degenerations allow us to complete the branching degenerations of $R(\delta_n)$ and
$R(i_{L,G})$ defined above to toric degenerations.

\begin{theorem}\label{diagtoric}
 The full tensor algebra $R(\delta_n)$ has a toric degeneration for each choice of the following objects.

\begin{enumerate}
\item A trivalent, oriented $n$-tree, $\tree$.\\
\item An ordering of the internal vertices $v \in V(\tree)$\\
\item An assignment $\bold{i}_v$ of strings to internal vertices of $\tree.$
\end{enumerate} 
\end{theorem}

\noindent
The resulting rational cone $C_{\tree}(\bold{i})$ is a fiber product of cones $C_3(\bold{i}_v)$
over copies of $\Delta_G$ according to the topology of $\tree.$  An element of this cone is a labeling of $\tree$ by dominant weights on the edges of $\tree$ and compatible dual canonical basis elements on the trinodes.  The map which forgets all data except the labels on the leaves is a map on cones $C_{\tree}(\bold{i}) \to C(\delta_n)$ which enriches the tensor branching cone.

\begin{figure}[htbp]
\centering
\begin{xy}
(-15, 4)*{\lambda_0};
(12, 10)*{\lambda_1};
(38, -2)*{\lambda_2};
(35, -14)*{\lambda_3};
(12, -1)*{\eta};
(0,5)*{b_{\lambda_0, \lambda_1, \eta, \vec{s}}};
(25, -3 )*{b_{\eta, \lambda_2, \lambda_3, \vec{t}}};
(0,0)*{\bullet} = "AA";
(30,15)*{\bullet} = "BB";
(50,-5)*{\bullet} = "CC";
(40,-25)*{\bullet} = "DD";
(-30,0)*{\bullet} = "EE";
(25,-7)*{\bullet} = "FF";
"BB"; "AA";**\dir{-}? >* \dir{>};
"AA"; "EE";**\dir{-}? >* \dir{>};
"FF"; "AA";**\dir{-}? >* \dir{>};
"DD"; "FF";**\dir{-}? >* \dir{>};
"CC"; "FF";**\dir{-}? >* \dir{>};
\end{xy}
\caption{A member of $C_{\tree}(\bold{i})$}
\end{figure}

These degenerations are all $T^{n+1}$ invariant with respect to the action which grades $R(\delta_n)$ by the tuple of dominant weights in each tensor product, so we also get toric degenerations of every subalgebra $R(\delta_n, M)\subset R(\delta_n)$ given as the sum of the graded pieces for a submonoid $M \subset \Delta_G^{n+1}.$

\begin{equation}
R(\delta_n, M) = \bigoplus _{\vec{\lambda} \in M} Hom_G(V(\lambda_0), V(\lambda_1) \otimes \ldots \otimes V(\lambda_n))\\
\end{equation}

\noindent
In particular, we get a degeneration of the algebra corresponding to the monoid of non-negative integer multiples of a particular tuple of weights, 
$\Z_{\geq 0}\vec{\lambda} \subset \Delta_G^{n+1}.$  This algebra, $R(\delta_n, \vec{\lambda})$ is the projective coordinate ring associated to a line bundle
$\mathcal{L}(\vec{\lambda})$ on a configuration space of $G,$ defined as the following GIT quotient.

\begin{equation}
P_{\vec{\lambda}}(G) = G \ql [G/P(\lambda_0^*) \times \ldots \times G/P(\lambda_n)]\\
\end{equation}

\noindent
Here $P(\lambda_i)$ is the parabolic subgroup of $G$ corresponding to the face of $\Delta_G$ which contains $\lambda_i.$   

\begin{corollary}
Let $C_{\tree}(\bold{i}, \vec{\lambda})$ be the polytope obtained from $C_{\tree}(\bold{i})$ by fixing the leaf edge values at the weights $\lambda_0, \ldots, \lambda_n.$   The lattice points in $C_{\tree}(\bold{i}, \vec{\lambda})$ are in bijection with a basis of the space of global sections $H^0(P_{\vec{\lambda}}(G), \mathcal{L}(\vec{\lambda})).$
\end{corollary}

\noindent
These spaces have been studied in detail for $G = SL_2(\C),$ where $G/P(\lambda)$ is always the projective line, see \cite{HMSV}, \cite{HMM}, \cite{M1}.   We also have a similar statement for the branching algebra $R(i_{L,G}).$

\begin{theorem}\label{levitoric}
 For any choices of the following information we construct a toric degeneration of $R(i_{L,G}).$ 
\begin{enumerate}
\item A chain of Levi subgroups $L \subset L_1 \subset \ldots \subset L_k\subset G$\\
\item A choice of string $\bold{i}_j$ for each Levi $L_j$\\
\end{enumerate} 

\end{theorem}
\noindent

 The result is easiest to describe when the strings $\bold{i}_j$ 
are all compatible in the sense that they concatenate to a string $\bold{i}$ of $G.$   For choices of these "adapted" strings, $R(i_{L,G})$ degenerates to the toric algebra  $\C[C_{\vec{I}}(\bold{i})].$
Here  $C_{\vec{I}}(\bold{i})$ is the fiber product  $C_{L}(\bold{i}_1) \times_{\Delta_{L_1}} \ldots \times_{\Delta_{L_m}} C_{L_m}(\bold{i}_G),$ where $C_{L_j}(\bold{i}_j)$ is the cone for the inclusion of $L_{j-1}$ into $L_j.$  There is also a map $C_{\vec{I}}(\bold{i}) \to C(i_{L,G})$ which enriches the Levi branching cone.

\subsection{The case of type $A$ and remarks}

In Section \ref{examples} we discuss a particular realization of $\C[C_{\tree}(\bold{i})]$ for $SL_m(\C)$ using Berenstein-Zelevinsky triangles \cite{BZ3}.  The result is a description of a particular $C_{\tree}(\bold{i})$ as the non-negative integer points in the intersection of a collection of linear spaces.

In certain cases the two constructions presented here apply to the same algebras.  As with $R(\delta_n),$ any subalgebra
$R(i_{L,G}, M) \subset R(i_{L,G})$ corresponding to a submonoid $M \subset C(i_{L,G}) \subset \Delta_H \times \Delta_G$ inherits the degenerations defined
by a chain of Levi subgroups $i_{L_1,L_2}, \ldots, i_{L_m, G}$ with selected strings $\bold{i}_k.$  An example of such a submonoid
is given by the multiplies of the $m-th$ dominant weight $\omega_m$ of $GL_n(\C).$  The subalgebra
$R_G(\Z_{\geq 0} \omega_m) \subset R_G$ is the projective coordinate ring of the Grassmannian $Gr_m(\C^n).$ 
This algebra can also be realized as a subalgebra $R(\delta_n, (\Z_{\geq 0}\omega_{m-1})^n) \subset R(\delta_n)$ for
$G = GL_m(\C).$   It would interesting to see how the different degenerations constructed
here for $R_G$ and $R(\delta_n)$ are related under these dualities. 

Valuations have been used to connect the combinatorics of branching problems to the commutative algebra of branching algebras before.
Howe, Tan and Willenbring \cite{HTW2} arrive at a SAGBI (Sub-algebra Analogue of Gr\"obner Basis for Ideals), 
interpretation of the Littlewood-Richardson rule, this is in the same spirit as this paper, as SAGBI degenerations  arise from the higher rank valuations implicit in the lexicographic ordering on the dual canonical basis.   Also in a similar spirit, Kaveh and Anderson\cite{K}, \cite{A}, have been 
connecting the dual canonical basis to a special type of valuation built from a full flag of subspaces in a flag 
variety or Schubert variety of $G.$  This allows them to realize meaningful polytopes from representation theory 
as the so-called Okounkov bodies of these valuations.  It would be interesting to see a similar interpretation 
the polytope $C_{\tree, \vec{\lambda}}(\bold{i})$ as an Okounkov body for some flag of subspaces in the 
projective variety $P_{\vec{\lambda}}(G).$

\section{Valuations and tropical geometry}\label{trop}

In this section we collect a few technical facts about valuations on a commutative algebra $A$ and their relationship to the tropical geometry of the associated variety $Spec(A)$.   We will define subductive generating sets $X \subset A$ with respect to a given valuation $v:A \to Q.$
Subductive sets are a generalization of SAGBI bases, and are an interesting construction in their own right.
While writing this paper we were made aware of the paper \cite{K}, where Kaveh has arrived
at the same definition.  We will use subductive sets to extend some of the results in \cite{C} and \cite{AB}
on constructing flat degenerations from term orderings on an algebra, in particular we will construct a cone
$D(X, \Gamma)$ in $\mathbb{V}_{\trop}(A)$ with respect to a subductive set $X$ and 
a generating set of relations $\Gamma.$    When this construction
is applied to the term orders defined by the dual canonical basis on $R(i_{L,G})$ and $R(\delta_n)$ it provides a way to to turn these term orders into flat degenerations.

\begin{definition}
We say $Q$ is a tropical field, if it is a totally ordered abelian group with a least element $-\infty.$
\end{definition}

\noindent
 For any tropical field $Q$ and an index set $X,$ we can form the polynomial semi-algebra $Q[X].$ An element in $Q[X]$ is a tropical polynomial. 

$$
\bigoplus q_i \otimes x_1^{\otimes n_1(i)} \otimes \ldots \otimes x_r^{\otimes n_r(i)}
$$

\noindent
Each tropical polynomial $F \in Q[X]$ determines a subspace $tr(F) \subset Q^X,$ the tropical
variety of $F.$  This is the set of points $\vec{q} \in Q^X$ where at least two monomials from $F$ take the maximum value. 

\begin{definition}
A valuation into $Q$ on a commutative algebra $A$ over $\C,$ is a function 
$v: A \to Q$ which satisfies the following properties. 

\begin{enumerate}
\item $v(ab) = v(a) \otimes v(b)$\\
\item $v(a + b) \leq v(a) \oplus v(b)$\\
\item $v(0) = -\infty$\\
\item $v(C) = 0, C \in \C^{\times}$\\
\end{enumerate}

We denote the set of all such functions by $\mathbb{V}_Q(A).$
\end{definition}

Let $v:K \to Q$ be a valuation on a field $K.$  
For a polynomial $f(X)  = \sum C_i \vec{x}^{\vec{m}_i} \in K[X]$ we define a polynomial
in the semiring $Q[X]$ called  the tropicalization of $f(X)$ as follows.

\begin{equation}
T(f) = \bigoplus v(C_i) \vec{x}^{\otimes \vec{m}_i}\\
\end{equation}

\noindent
For an ideal $I \subset K[X]$  we can then define the tropical variety $tr(I) = \cap_{f \in I} tr(T(f)).$
Notice that if $v: K \to Q$ is the trivial valuation, then $T(f)$ has only the tropically multiplicative identity $0$ as a coefficient.

When $Q = \trop,$ the set $\mathbb{V}_{\trop}(A)$ has the structure
of a Haussdorf topological space, and it is referred to as the analytification
of $Spec(A),$ see \cite{P}, \cite{B}.  In this case, a tropical variety
$tr(I) \subset \trop^{X}$ is a subfan of the Gr\"obner fan of $I.$
Higher rank tropical fields appear naturally in combinatorial commutative algebra as well.

\begin{example}
In the setting of SAGBI theory, one studies a subalgebra $A \subset \C[X]$ of
a polynomial algebra with the deg-lex term order.   This defines a valuation into $\Z^{|X|+1},,$ ordered lexicographically.  In particular, a polynomial $f(x)$ is taken to its degree paired with the exponents of its top monomial. 
\end{example}

\noindent
There is a close relationship between the sets $tr_Q(I)$ and $\mathbb{V}_Q(A).$

\begin{theorem}
Let $I$ be a presenting ideal of $A.$

$$
\begin{CD}
0 @>>> I @>>> \C[X] @>>> A @>>> 0\\
\end{CD}
$$
\noindent
Then there is a map,

\begin{equation}
\pi_X: \mathbb{V}_Q(A) \to tr_Q(I)\\
\end{equation}

\noindent
furthermore, there is a bijection.

\begin{equation}
\mathbb{V}_Q(A) \cong \varprojlim_{X \subset A} tr_Q(I)\\
\end{equation}
\end{theorem}

\begin{proof}
For this, emulate the proof in \cite{P} of Theorem 1.1.
\end{proof}

\subsection{The subduction algorithm}

In SAGBI theory the subduction algorithm writes an arbitrary
polynomial $p(x) \in A \subset \C[X]$ as a polynomial in a SAGBI basis $\{f_1, \ldots, f_r\} \subset A,$
see for example \cite{St}, Chapter 11. We will give a generalization of this algorithm to a general valuation on an algebra $A.$  

\begin{definition}
For a valuation $v:A \to Q,$ filter $A$ by values of $Q,$ $A= \bigcup A_{\leq q}.$ 
Let $gr_v(A)$ be the associated graded algebra with respect to this filtration.
let $X \subset A$ give a homogenous generating set $\bar{X}  \subset gr_v(A).$ For $a \in A$ we perform subduction on $a$ as follows.  

\begin{enumerate}
\item Find $p(\bar{x}) = \bar{a} \in gr_v(A).$ \\
\item Consider $a - p(x) \in A,$ which must have strictly lower filtration value. Repeat (1) on $a- p(x).$\\
\end{enumerate}

\noindent
If this process always terminates then we say $X \subset A$ is subductive.
\end{definition}


\noindent
Any subductive set $X$ generates $A,$ also notice that $X$ subductive implies that $X \cup \{y\}$ for $y \in A$ subductive.

\begin{definition}
For a vector $\vec{q} \in Q^X$ and a polynomial $f \in \C[X],$
we define the initial form $in_{\vec{q}}(f)$ to be the sum of the monomials in $f$ which have highest value on $\vec{q}$ when tropicalized.   For $\vec{q} \in tr_Q(I)$ define $in_{\vec{q}}(I)$ to be the ideal generated by the $\vec{q}-$initial forms from $I.$
\end{definition}

Each ideal $in_{\vec{q}}(I)$  defines a new algebra $\C[X]/in_{\vec{q}}(I).$
For any $v \in \mathbb{V}_Q(A),$ both $v$ and $\pi_X(v)$ define algebras
constructed from $A,$ which leads naturally to the following proposition.

\begin{proposition}\label{subduct}
If $X \subset A$ is subductive with respect to $v$ then we have the following. 

\begin{equation}
gr_v(A) \cong \C[X]/in_v(I)\\
\end{equation}
\end{proposition}

\begin{proof}
Let $J$ be the homogeneous ideal which presents $gr_v(A)$ as a quotient of $\C[X],$
we will show $J = in_v(I).$  The inclusion $in_v(I) \subset J$ is easier, any $f \in in_v(I)$ is of the form
$in_v(f + f')$ for $f'$ of lower $v-$degree in $\C[X],$ as graded by
$v.$  This implies that $f(\vec{x}) \in A_{< r}$ for $r$ the $v$-degree of the
monomials in $f,$ so $f(\vec{x}) = 0$ in $gr_v(A).$  Now, if $g \in J,$ is homogeneous
then $g(\vec{x}) \in A_{< r}$ for $r$ the $v-$degree of $g,$ which means
there is a polynomial in monomial terms of strictly less $v-$degree, $p(\vec{x})$
such that $g(\vec{x}) - p(\vec{x}) \in A_{< r_1}$ for $r_1$ the $v-$degree
of $p_1$ and $r_1 < r.$ Continuing this way, we get a polynomial relation $\hat{g}$
with $in_v(\hat{g}) = g$ by the termination of the subduction algorithm. 
\end{proof}

\subsection{Partial orderings}

We make use of a number of partially ordered cones, specifically Weyl chambers, and products of Weyl chambers with lexicographically
ordered products non-negative integers.  We will cover some techniques for turning these partial orders
into valuations and material for flat degenerations.   Consider a commutative algebra $A = \bigoplus_{c \in C} A_c$ 
with underlying vector space graded by a commutative monoid $C.$  Suppose there is a partial ordering $<$ on $C$ such that
$A_{\eta_1}A_{\eta_2} \subset \bigoplus_{\eta \leq \eta_1 + \eta_2} A_{\eta}.$ We can define an associated
graded algebra by altering the multiplication operation on the vector space $\bigoplus_{c\in C} A_c$ by projecting
onto the top weighted component. 

\begin{equation}
A_c \times A_d \to A_{c + d} \\
\end{equation}

\noindent
We call this new algebra $gr_{<}(A).$  For a tropical field $Q$  let $e: C \to Q$
be a monoidal map such that $c < d$ implies $e(c) \oplus e(d) = e(d)$ in $Q.$

\begin{proposition}\label{pval}
If $gr_{<}(A)$ is a domain, any $e$ as above defines a valuation on $A.$
\end{proposition}

\begin{proof}
We begin by assuming that $c < d$ implies that $e(c) < e(d).$  We claim
that if this is the case, then $gr_e(A) = gr_{<}(A).$  The function $e$ defines
a filtration on $A$ by $A_{\leq r} = \bigoplus_{e(c) \leq r} A_c$
and $A_{< r} =  \bigoplus_{e(c) < r} A_c.$  Clearly this is multiplicative
and the associated graded algebra can be identified as follows. 

\begin{equation}
gr_e(A) = \bigoplus_{r \in Q} A_{\leq r}/A_{ < r} = \bigoplus_{r\in Q} (\bigoplus_{e(c) = r} A_c)\\
\end{equation}

\noindent
For $\sum a_j \in A_r,$ and $\sum b_k \in A_s,$ where each $a_j$ and $b_k$ are in different
$C-$components,  we have $(\sum a_j)(\sum b_k) = \sum \bar{a_jb_k},$ where 
$\bar{a_jb_k} \in A_{\leq e(j+k)} / A_{< e(j+k)}.$ But since $e(c) < e(d)$ whenever $c < d$ 
this is just multiplication in $gr_{<}(A),$ and therefore this element is non-zero.  If $e$ only satisfies
$e(c) \leq e(d)$ then the result of multiplication is the image of $a_jb_k \in A_{j + k}$ plus more, which cannot cancel $a_jb_k$ because it is in a lower summand by assumption.  This implies that $e$ defines a filtration with associated graded algebra a domain, and hence a valuation. 
\end{proof}

Note that the set $\{ e: C \to Q | c < d \rightarrow e(c) \oplus e(d) = e(d) \}$ forms a cone $C_Q \subset Hom(C, Q).$  
The above proposition implies that there is a map $C_Q \to \mathbb{V}_{Q}(A).$  The following proposition
shows that generating sets of the associated graded algebra are subductive with respect to these valuations.

\begin{proposition}\label{psubduct}
Let $A$ be filtered by a partially ordered monoid as above, suppose that the 
down-set $D_c = \{ d \in C | d < c\}$ is finite for all $c,$ and let $e: C \to Q$ be as
above.  If $\bar{X} \subset gr_e(A)$ is a homogeneous generating set, then it lifts to a
subductive set $X$ with respect to $e.$
\end{proposition}

\begin{proof}
Given $a \in A$ we may assume without loss of generality that $a \in A_c$ for some $c.$ Let
$p(\bar{x}) = a$ in the associated graded algebra.  We have that $a - p(x)$ is composed of terms of weight strictly less than $c.$
Since $D_c$ is finite, we only need repeat this process a finite number of times, so the subduction algorithm terminates. 
\end{proof}

We will prove a strengthening of Proposition \ref{pval}. 
Let $X \subset A$ give a homogeneous generating set for $gr_<(A),$ and let
$J$ be the homogeneous ideal of the corresponding presentation.  Using a subduction type argument
we can lift an element $g \in J$ to $\hat{g} \in I$ with the property that $\hat{g} = g + \ell$
with all the terms in $\ell$ of weight less than the homogenous degree of $g.$  Let $\Gamma$
be a lifting of a generating set of $J.$

\begin{definition}
Define $D(X, \Gamma)$  to be the set of monoidal maps $e:C \to Q$ such that 
$e$ weights the monomials of $g$ higher than any term in $\ell$ for all $\hat{g} \in \Gamma.$
\end{definition}

\noindent
We define a filtration on $A$ by letting $A_{\leq q} = \C\{ x_1^{m_1}\ldots x_k^{m_k} | \sum m_i e(c_i) \leq q\},$ where
$c_i$ is the $C-$degree of $x_i.$ This is the span of all monomials in $X$ with $e$-degree less than or equal to $q.$

\begin{theorem}\label{pval2}
Let $A$ be as above, with $D_c$ finite for all $c.$  If $e \in D(X, \Gamma),$
then $e$ defines a valuation from $A$ to $Q.$
\end{theorem}

\begin{proof}
The set $X$ is subductive for the filtration defined by $e,$ so we have
$gr_e(A) = \C[X]/in_e(I).$  We will show that $in_e(I) = J,$ which establishes
that the associated graded algebra of $e$ is $gr_<(A),$ a domain.   One direction
is clear by definition. If $f \in J$ then $f = \sum x^{M_i}g_i$ for $g_i + \ell_i \in G,$
by assumption these are all initial forms for $e.$

For a general $f \in I$ we may write $f = \sum f_c$ for $c \in C.$  Consider
the set $D(f) = \bigcup D_c,$ the union of all the finite down-sets of the weights of the 
terms in the sum above, this is a finite set by assumption.  Consider a maximal element $c_m \in D(f),$ and $f_{c_m}(x) \in A.$  This must land
in the sum $\bigoplus_{c \in D_{c_m}} A_c.$  If a term of $f_{c_m}(\vec{x})$ is non-zero in $A_{c_m},$ then it cannot cancel with any terms from $f_c(\vec{x})$ for any other $c \in D(f),$ 
which contradicts $f \in I,$ this implies that $f_{c_m}(\vec{x}) \in \bigoplus_{c < c_m} A_c,$
and therefore $f_{c_m} \in J.$  We can choose a homogeneous expression $f_{c_m} = \sum x^{M_i}g_i$
and consider $f - \sum x^{M_i}(g_i + \ell_i) \in I.$  By construction, this element has the property that the union of all of the down sets of its weights is a strict subset of $D(f)$.  After repeating this process a finite number of steps we obtain an expression $f = \sum x^{M_i}(g_i + \ell_i)$ where the terms $x^{M_i}g_i$ for a fixed weight don't cancel.  

This shows two things.  First of all, it proves that $G$ generates $I,$ and it
shows that for any lift $\hat{f}$ of an initial form $f \in in_e(I)$ we can produce an expression as above, take its initial form, and recover $f$ as a sum of terms of the form $x^{M_i}g_i,$ which are all in $J.$
\end{proof}

The purpose of these last two propositions is to allow us to start with an algebra that has been filtered by a partially ordered cone, and turn that partial ordering into a valuation.  If that valuation takes values in $\trop$ then we can say more geometrically.  

\begin{proposition}\label{reese}
If $v:A \to \trop$ takes non-negative integer values, then $gr_v(A)$ is a flat
degeneration of $A.$
\end{proposition}

\noindent
This follows from the standard construction of the Reese algebra, see for example Proposition 2.2 of \cite{AB}.  For a filtration $F$ on $A$, one constructs a new algebra $\mathcal{R} = \C[t] \otimes (\bigoplus t^n F_{\leq n}),$ this algebra is flat over $\C[t]$ by construction because $t$ is not a zero divisor. specializing at $t = 0$ produces $gr_v(A)$ and $t \neq 0$ produces $A.$  

In this paper, the constructions using the dual canonical basis require us to use Proposition \ref{pval2} to turn a filtration into a flat degeneration.  In order to do so, we require an $e \in D(X, \Gamma)$ as above.  When the filtration is by a partially ordered rational cone, as is the case with all of our constructions, we can use the argument for Proposition 2.2 in \cite{AB} to create this element.

\section{Algebraic constructions from branching problems}\label{branch}

In this section we construct the filtration by factorization diagrams on $R(\phi),$  and employ
Proposition \ref{pval} to build cones of valuations $B(\vec{\phi}) \subset \mathbb{V}_{\trop}(R(\phi))$ for every factorization $\vec{\phi}$ of $\phi.$  These are then glued together to form the complex $K_{\phi} \subset \mathbb{V}_{\trop}(R(\phi)).$    We describe a set of elements in $K_{\phi}$ which define flat degenerations of $R(\phi),$ and we will use these constructions on $R(i_{L,G})$ and $R(\delta_n)$ as the first step in constructing toric degenerations of $R(\delta_n)$ and $R(i_{L,G})$.

\subsection{Filtrations and valuations from a factorization}

Recall that the algebra $R(\phi)$ is multigraded by pairs of dominant weights $(\gamma, \lambda) \in \Delta_H \times \Delta_G.$  
The $(\gamma, \lambda)$ component is isomorphic to $Hom_H(\C, V(\gamma) \otimes V(\lambda)) = Hom_H(V(\gamma^*), V(\lambda)).$  Recall that $C_*: V(\lambda) \otimes V(\eta) \to V(\lambda +\eta) $ is the Cartan multiplication operation, there is a dual operation
$C^*: V(\lambda + \eta) \to V(\lambda) \otimes V(\eta)$ defined by sending the highest weight vector $v_{\lambda + \eta}$ to $v_{\lambda}\otimes v_{\eta}.$  The lemma below appears in \cite{M3}, and its proof is a straightforward computation.

\begin{lemma}
When the multiplication of two homogeneous elements $\phi_1, \phi_2 \in R(\phi)$ is computed by the following diagram. 

$$
\begin{CD}
V(\gamma_1^* + \gamma_2^*) @>C^*>> V(\gamma_1^*)\otimes V(\gamma_2^*) @> \phi_1 \otimes \phi_2>> V(\lambda_1)\otimes V(\lambda_2) @>C_*>> V(\lambda_1 + \lambda_2)\\
\end{CD}
$$

\end{lemma}

Now consider a factorization $\phi = \psi \circ \pi$ through 
a third reductive group $K.$ For each graded piece $Hom_H(V(\gamma^*), V(\lambda))$
this gives a direct sum decomposition. 

\begin{equation}
Hom_H(V(\gamma^*), V(\lambda)) = \bigoplus_{\tau \in \Delta_K} Hom_H(V(\gamma^*), V(\tau^*)) \otimes Hom_K(V(\tau^*), V(\lambda))\\
\end{equation}

\noindent
Each element in this decomposition can be pictured as a factorization diagram. 

$$
\begin{CD}
V(\gamma^*) @>f>> V(\tau^*) @>g>> V(\lambda)\\
\end{CD}
$$

\noindent
Here $f$ is $H$-linear and $g$ is $K$-linear, so we can also regard this as an element of
$R(\pi) \otimes R(\psi)$.  Now we can write the multiplication in $R(\phi)$ in terms of these diagrams.
$$
\begin{CD}
V(\gamma_1^*)\otimes V(\gamma_2^*) @> f_1 \otimes f_2>> V(\tau_1^*)\otimes V(\tau_2^*) @> g_1 \otimes g_2>> V(\lambda_1) \otimes V(\lambda_2)\\
@AC^* AA  @.    @V C_* VV\\
V(\gamma_1^* + \gamma_2^*) @.  @. V(\lambda_1 + \lambda_2)\\
\end{CD}
$$

\noindent
The middle component has a direct sum decomposition as a $K$ module, 

$$V(\tau_1) \otimes V(\tau_2) = \bigoplus I_{\eta} \otimes V(\eta).$$

\noindent
There are projection $P_{\eta}$ and injection $Q_{\eta}$ maps for each component
in this decomposition.  In this way, the product element can be written as follows. 

\begin{equation}
C_* \circ[g_1 \otimes g_2] \circ [f_1 \otimes f_2] \circ C^*
= \sum C_* \circ[g_1 \otimes g_2] \circ P_{\eta} \circ Q_{\eta} \circ [f_1 \otimes f_2] \circ C^*\\
\end{equation}

\noindent
The $\tau_1^* + \tau_2^*$ component in this sum is exactly $C_* \circ[g_1 \otimes g_2] \circ C^* \circ C_* \circ [f_1 \otimes f_2] \circ C^*.$  By the above observations
we have the following. 

\begin{equation}\label{branchfilter}
W(\alpha_1, \eta_1, \beta_1)W(\alpha_2, \eta_2, \beta_2) \subset \bigoplus_{\eta \leq \eta_1 + \eta_2} W(\alpha_1 + \alpha_2, \eta, \beta_1 + \beta_2)\\
\end{equation}

 Using the direct sum decomposition above, we can filter $R(\phi) = \bigoplus W(\alpha, \eta, \beta)$  by the ordering on the dominant weights of $K.$  Using the filtering index, the $(\alpha_1, \eta_1, \beta_1)$ and $(\alpha_2, \eta_2, \beta_2)$ components of $R(\phi)$ multiply to give elements of weight $(\alpha_1 + \alpha_2, \eta_1 + \eta_2, \beta_1 + \beta_2)$ and lower.  Furthermore, if we cut off the components of strictly lower weight, we get the multiplication operation in the ring $R(\pi) \otimes R(\psi).$  This proves the following version of theorem \ref{branchfilter}. 

\begin{proposition}
There is a filtration $F_{\pi, \psi}$ on $R(\phi)$ by factorization diagrams
 associated to each factorization $\phi = \pi \circ \psi,$
such that the associated graded algebra $gr_{\pi, \psi}(R(\phi))$ is isomorphic to 
$[R(\phi) \otimes R(\pi)]^{T_K}.$ 
\end{proposition}

An element $\rho \in \Delta^*,$ the dual Weyl chamber
is defined by its non-negativity on positive roots.  This implies that the values obtained by applying $\rho$ respect the partial ordering  on the dominant weights of $K$. The next corollary follows from Proposition \ref{pval}.

\begin{corollary}
To each factorization $\pi \circ \psi = \phi$ in the category of reductive groups there is a map $\hat{f}_{\pi,\psi}: \Delta_H^*\times \Delta_K^* \times \Delta_G^* \to \mathbb{V}_{\trop}(R(\phi)).$
\end{corollary}

\begin{proposition}
The image of the above map is the quotient of the cone $\Delta_H^*\times \Delta_K^* \times \Delta_G^*$ by
the relation $X \sim Y$ when $X(\omega) = Y(\omega)$ for all $\omega \in C(\pi) \times_{\Delta_K} C(\psi) = C(\pi, \psi),$ and is therefore equal to $B(\psi, \pi).$
\end{proposition}

\begin{proof}
The lattice points in the cone $ C(\pi) \times_{\Delta_K} C(\psi)$ give exactly the triples $(\alpha, \beta, \gamma)$ such
that the graded component $W(\alpha, \eta, \beta) \subset R(\pi\circ \psi)$
are non-zero.  Since the filtrations defined by $\Delta_H^*\times \Delta_K^* \times \Delta_G^*$ are defined from the weight information
from these components, the proposition follows.
\end{proof}

By construction the filtration by factorization diagrams is $T_H \times T_G$ linear, so it passes
to subrings of $R(\phi)$ which are graded by a submonoid of the dominant weights in $\Delta_H \times \Delta_G.$  For example, the submonoid $\{0\} \times \Delta_G \subset \Delta_H \times \Delta_G$ corresponds to the invariant ring $R_G^H,$ given by the factorization 
diagrams where the $H$-weight is trivial. Every factorization  $\phi = \pi \circ \psi$ gives a filtration on $R_G^H$  with associated graded algebra $[R_K^H \otimes R(\pi)]^{T_K}.$  Note also that $[R(\phi) \otimes R(\pi)]^{T_K}$
has an action of $T_K,$ so the degeneration adds torus symmetries in accordance with the corresponding factorization.  We can extend
this construction to any finite factorization of a morphism $\phi:H \to G.$  
$$
\begin{CD}
H  @>\phi_1>> K_1 \ldots K_{n-1} @>\phi_n>> K_n\\
\end{CD}
$$

\noindent
We write $R(\phi)$ as a direct
sum of spaces of factorization diagrams as below, where $\lambda_i$ is a dominant weight of $K_i.$   This gives a filtration on $R(\phi)$ by the same argument as above, we define the spaces $W(\vec{\phi}, \vec{\lambda})$ as above.  

\begin{equation}
W(\vec{\phi}, \vec{\lambda}) = Hom_H(V(\lambda_0), V(\lambda_1)) \otimes \ldots \otimes Hom_{K_{n-1}}(V(\lambda_{n-1}), V(\lambda_n))\\
\end{equation}

We can now apply Proposition \ref{pval}, and Proposition \ref{reese} implies that coweight vectors which give positive integer values define flat degenerations of $R(\phi).$  Furthermore, when the coweight vector $\vec{\rho}$ is chosen to have strictly positive contributions from all coroots, then $\vec{\tau}_1 < \vec{\tau}_2$ as dominant weights implies $\vec{\rho}(\vec{\tau}_1) < \vec{\rho}(\vec{\tau}_2),$ so the flat degeneration is isomorphic to $[R(\phi_1)\otimes \ldots \otimes R(\phi_m)]^{T_1 \times \ldots \times T_m}.$

\subsection{The complex $K_{\phi}$}

We recall the degeneracy and face maps from the introduction.   

\begin{equation}
(d_k^i)^*:B(\phi_1, \ldots, \phi_i \circ \phi_{i-1}, \ldots \phi_k) \to B(\phi_1, \ldots, \phi_k)\\ 
\end{equation}

\begin{equation}
(d_k^i)^*: (\rho_0, \ldots, \rho_k) \to (\rho_0, \ldots, 0, \ldots, \rho_k)\\
\end{equation}

\begin{equation}
(s_k^i)^*: B(\phi_1, \ldots, \phi_i, id, \phi_{i+1}, \ldots, \phi_k) \to  B(\phi_1, \ldots, \phi_i, \phi_{i+1}, \ldots, \phi_k)\\ 
\end{equation}

\begin{equation}
(s_k^i)^*:(\rho_0, \ldots,\rho_i, \rho_{i+1},\ldots, \rho_k) \to (rho_0, \ldots, \rho_i + \rho_{i+1}, \ldots, \rho_k)\\
\end{equation}

\noindent
In order to prove Theorem \ref{valcom} we must show these maps are well-defined, and commute with the $f_{\vec{\phi}}: B(\vec{\phi}) \to \mathbb{V}_{\trop}(R(\phi)).$   To do this we define corresponding maps on the branching cones
$C(\vec{\phi}).$  We will see shortly that the operations of composing two consecutive morphisms
in a diagram corresponds to the following  degeneracy map on branching cones. Let $\Delta_{\vec{\phi}}$ be the product of Weyl chambers for the groups in the factorization defined by $\vec{\phi}.$
$$
\begin{CD}
\Delta_{\vec{\phi}} @>(d_k^i)_*>> \Delta_{d_k^i(\vec{\phi})}\\
@AAA @AAA\\
C(\vec{\phi}) @>(d_k^i)_*>> C(d_k^i(\vec{\phi}))\\
\end{CD}
$$
\begin{equation}
(d_k^i)_*(\lambda_0, \ldots, \lambda_k) = (\lambda_0, \ldots, \hat{\lambda_i}, \ldots, \lambda_k)\\
\end{equation}

\noindent
On branching diagrams, these are the operations of composing two adjacent morphisms and adding
in an identity morphism, respectively. Similarly, the operation of inserting an identity into a diagram of morphisms corresponds to the following
face map on branching cones. 

$$
\begin{CD}
\Delta_{\vec{\phi}} @>(s_k^i)_*>> \Delta_{s_k^i(\vec{\phi})}\\
@AAA @AAA\\
C(\vec{\phi}) @>(s_k^i)_*>> C(s_k^i(\vec{\phi}))\\
\end{CD}
$$
\begin{equation}
(s_k^i)_*(\lambda_0, \ldots, \lambda_k) = (\lambda_0, \ldots, \lambda_i, \lambda_i, \ldots, \lambda_k)\\
\end{equation}

\noindent
Since branching valuations are determined by the branching
cones by definition, showing that these maps are well-defined on the $B(\vec{\phi})$ and
commute with the maps $f_{\vec{\phi}}$ amounts to proving the following
formulas. For any $\vec{\lambda} \in \Delta_{\vec{\phi}}$ and $\vec{\rho} \in \Delta_{\vec{\phi}}^*,$

\begin{equation}
(d_k^i)^*(\vec{\rho})(\vec{\lambda}) =  \sum_{j \neq i} \rho_j(\lambda_j) + 0\lambda_i = \vec{\rho}((d_k^i)_*(\vec{\lambda}))\\
\end{equation}

\begin{equation}
(s_k^i)^*(\vec{\rho})(\vec{\lambda}) = \rho_1(\lambda_1) + \ldots + \rho_i(\lambda_i) + \rho_{i+1}(\lambda_i) + \ldots + \rho_k(\lambda_{k-1}) = \vec{\rho}((s_k^i)_*(\vec{\lambda}))\\
\end{equation}

\noindent
If $\vec{\rho}$ and $\vec{\rho}' \in \Delta_{d_k^i(\vec{\phi})}^*$ agree on all $\vec{\lambda} \in C(d_k^i(\vec{\phi}))$ 
then for any $\vec{\gamma} \in C(\vec{\phi})$ we have $\vec{\rho}((d_k^i)_*(\vec{\gamma}))= \vec{\rho}'((d_k^i)_*(\vec{\gamma}))$
and therefore $(d_k^i)^*(\vec{\rho})(\vec{\gamma}) = (d_k^i)^*(\vec{\rho}')(\vec{\gamma}).$  This shows that the map
$(d_k^i)^*: B(d_k^i(\vec{\phi})) \to B(\vec{\phi})$ is well-defined.  A similar argument proves that $(s_k^i)^*: B(s_k^i(\vec{\phi})) \to B(\vec{\phi})$
is well-defined.   Now we can check that the maps $(d_k^i)^*$ and $(s_k^i)^*$ commute with the maps $f_{\vec{\phi}}$ to $\mathbb{V}_{\trop}(R(\phi)).$
We have $R(\phi) = \bigoplus_{\vec{\lambda} \in C(\vec{\phi})} W(\vec{\phi}, \vec{\lambda}),$ as above.  The operation $d_k^i(\vec{\phi})$ collapses away the weights for $K_i,$ one checks that we get

\begin{equation}
W(d_k^i(\vec{\phi}), \vec{\gamma}) = \bigoplus_{\vec{\eta} \in C(\vec{\phi})| (d_k^i)_*(\vec{\eta}) = \vec{\gamma}} W(\vec{\phi}, \vec{\eta})\\
\end{equation}

\noindent
Any element $\rho \in \Delta_{d_k^i(\vec{\phi})}^*$ weights everything on the right hand side above the same, which matches
the way $(d_k^i)^*(\rho)$ weights the left hand side.  This shows that they define the same valuation.  Similarly, we have

\begin{equation}
W(s_k^i(\vec{\phi}), \vec{\gamma}) = \ldots \otimes Hom_{K_i}(V(\gamma_i), V(\gamma_{i+1})) \otimes Hom_{K_i}(V(\gamma_{i+1}), V(\gamma_{i+2})) \otimes \ldots \\
\end{equation}

\noindent
But this is only non-zero when $\gamma_{i+1} = \gamma_i,$ in which case $Hom_{K_i}(V(\gamma_i), V(\gamma_{i+1})) = \C,$ and we 
get $W(s_k^i(\vec{\phi}), \vec{\gamma}) = W(s_k^i(\vec{\phi}), (s_k^i)_*(\vec{\gamma}')) \cong W(\vec{\phi}, \vec{\gamma}')$
where $\vec{\gamma}$ has no $i$-th coweight, $(\gamma_1, \dots, \gamma_i, \gamma_{i+2}, \ldots, \gamma_{k+1}).$
The element $(s_k^i)^*(\rho_0, \ldots, \rho_i, \rho_{i+1}, \ldots, \rho_{k+1}) = (\rho_0, \ldots, \rho_i + \rho_{i+1}, \ldots, \rho_{k+1})$
has the same value on the elements of $W(\vec{\phi}, \vec{\gamma}')$ as $\vec{\rho}$ does on $ W(s_k^i(\vec{\phi}), (s_k^i)_*(\vec{\gamma}')) ,$
so they also define the same valuation.  Since $K_{\phi}$ is defined as the colimit of the $B(\vec{\phi})$ over the maps $(d_k^i)^*$ and $(s_k^i)^*,$ this proves Theorem \ref{valcom}.

\subsection{General $G-$algebras and $G-$actions on valuations}

We briefly describe how to construct a map from $K_{\phi}$ to $\mathbb{V}_{\trop}(A^H)$
for $\phi:H \to G$ and any $G-$algebra $A.$  We also discuss the action of $G$ on the complex $K_{i_G}$ for the map $i_G: 1 \to G.$   First we look at a generalization of Theorem \ref{valcom}.
   
\begin{theorem}
Let $A$ be a $G$-algebra, then there is a map $F_{\phi}:K_{\phi} \to \mathbb{V}_{\trop}(A^H).$
\end{theorem}

\begin{corollary}\label{G-act}
There is a map $F_{i_G}:K_{i_G} \to \mathbb{V}_{\trop}(A)$ for any algebra $A$ with a $G$ action.  
\end{corollary}

This is shown with very similar methods employed thus far in this paper.  One follows Chapter $3,$ Section $15$ of \cite{G} and notes that $A$ always carries a direct sum decomposition as a vector space. 

\begin{equation}
A = \bigoplus_{\lambda \in \Delta} Hom(V(\lambda), A)\otimes V(\lambda)\\
\end{equation}

\noindent
Grosshans shows that this is a $G-$stable filtration of algebras, with associated graded algebra $[A^{U_+} \otimes R_G]^T,$ where $U_+$ is a maximal unipotent.
Now the filtrations and valuations on $R_G^{U_H}$ constructed in Section \ref{branch} can be extended to $A^{U_H}$ and its $\Delta_H \times \Delta_G$-multigraded subalgebras, in particular $A^H.$  Corollary \ref{G-act} is of particular interest because the valuations $\mathbb{V}_{\trop}(A)$ of a $G-$algebra $A$ carry an action by the group $G$ by pre-composition. 

\begin{equation}
[g\circ v](f) = v( g^{-1}(f))\\
\end{equation}

\begin{theorem}\label{Gaction}
For any $v \in \mathbb{V}_{\trop}(A)$ coming from $K_{i_G},$ and $g \in G,$ the valuation $g\circ v$ is also in the image of $K_{i_G}.$
\end{theorem}

\begin{proof}
Let $v$ be defined by the coweights $(\rho_1, \ldots, \rho_m)$ on a chain of reductive
groups. 

$$
\begin{CD}
1 @>i_{K_1}>> K_1 @>\phi_1>> \ldots @>\phi_m>> G\\
\end{CD}
$$

\noindent
We define a new valuation $v_g$ by taking the same co-weights, and conjugating
$\phi_m$ by $g.$

$$
\begin{CD}
1 @>i_{K_1}>> K_1 @>\phi_1>> \ldots @>g\phi_m g^{-1}>> G\\
\end{CD}
$$

\noindent
We show that $g\circ v = v_g.$  For $K_i$ in this chain, the map $\phi_i\circ \ldots \circ g \phi_m g^{-1}$ is equal to $g[\phi_i \circ \ldots \circ \phi_m]g^{-1}.$  We consider the branching of an irreducible representation viewed as a subspace of $A$ with respect to this map. 

\begin{equation}
V(\lambda) = \bigoplus Hom_{K_i}(V(\eta), V(\lambda))\otimes V(\eta)\\
\end{equation}

\noindent
The element $g$ defines an isomorphism of $K_i$ modules $i_g: A \to gA,$ where $K_i$
acts on $A$ through $\phi_1 \circ \ldots \circ \phi_m$ and on $gA$ through $g\phi_1 \circ \ldots \circ \phi_m g^{-1}.$  Under this isomorphism, $gV(\lambda)$ decomposes as
$\bigoplus g Hom_{K_i}(V(\eta), V(\lambda))\otimes V(\eta).$  This implies the following equation.

\begin{equation}
v_g(g\circ f) = v(f)\\
\end{equation}

\noindent
This implies that $[g^{-1} \circ v_g](f) = v(f),$ so $v_g(f) = [g\circ v](f).$
\end{proof}

A consequence of this proof is that the isomorphism type of each part of the chain is preserved by the $G-$action as well.  In this way, the space $K_{i_G}$ decomposes into chambers depending on these isomorphism types.   

\begin{corollary}
Let $v:A \to \trop$ be a valuation from by the branching cone defined by the chain 

\begin{equation}
1 \to T \to G,\\
\end{equation}

then $g\circ v$ is a valuation from the branching cone defined by the chain

\begin{equation}
1 \to gTg^{-1} \to G.\\
\end{equation}
\end{corollary}

A valuation $v$ as above is a choice $\chi \in Hom(\C^*, T)\otimes \R$ which is then evaluated on the weights of $A$ via its eigen-decomposition.   Since any two maximal tori of 
$G$ are conjugate, this shows that the set of all branching valuations coming from maximal tori is covered by $G$-translates of those coming from a fixed maximal torus.  Similarly, the set of all branching valuations coming from a chain of Levi subgroups is covered by $G$-translates of the set of branching valuations coming from chains of Levi subgroups which share a common maximal torus, $H_T(G)$.

\begin{remark}
It is straightforward to show that $H_T(G)$ is a cone over the spherical building of $G.$
\end{remark}

\section{Degenerations and the dual canonical basis}

In this section we will give a brief overview of the labeling of the dual canonical basis
$B \subset R_G,$ by the string parameters associated to a decomposition $\bold{i} \in R(w_0)$ of the longest word in the Weyl group of $G.$   To begin, we assume that $G$ is simple.  Lusztig \cite{Lu} (also Kashiwara \cite{Ka}) constructed a canonically defined
basis $\mathbb{B}$ of a certain subalgebra $\mathcal{U}^+$ of the quantized universal enveloping algebra of the Lie algebra of $G,$ $\mathcal{U}_q(\mathfrak{g}).$
Specialization this basis at $q = 1$ yields a basis for each irreducible representation $V(\lambda)$ of $G,$ we denote the dual of this basis in $V(\lambda^*)$ by $B(\lambda^*).$

\subsection{String parameters}

  In \cite{BZ1}, Berenstein and Zelevinsky studied the parameterization of this basis by the string parameters associated to $\bold{i},$ which identifies $B(\lambda)$ with a subset of $\Z_{\geq 0}^N \subset \R^N.$  After summing over all $\lambda \in \Delta,$ there is a bijection of $B = \coprod_{\lambda \in \Delta} B(\lambda)$ with the lattice points in a rational cone $C(\bold{i}) \subset \Delta \times \Z_{\geq 0}^N.$  In order to describe
this cone we must first define the system for indexing weights in representations invented by Berenstein and Zelevinksy \cite{BZ1}
called $\bold{i}-$trails.  An $\bold{i}-$trail from a weight $\gamma$ to a weight $\eta$ in the weight polytope of a representation
$V$ is a sequence of weights $(\gamma, \gamma_1, \ldots, \gamma_{\ell-1}, \eta),$ such that consecutive differences of weights
are integer multiples of simple roots from $\bold{i}$, $\gamma_i - \gamma_{i+1} = c_k \alpha_{i_k},$ and the application 
of the raising operators $e_{i_1}^{c_1} \circ \ldots \circ e_{i_{\ell}}^{c_{\ell}}: V_{\eta} \to V_{\gamma}$ is non-zero. 
For any $\bold{i}-$trail $\pi,$ Berenstein and Zelevinsky define  $d_k(\pi) = \frac{1}{2}H_{\alpha_{i_k}}(\gamma_{k-1} + \gamma_k).$ 
In what follows, the entries of the Cartan matrix $A$ are denoted $a_{ij}.$

\begin{definition}
The cone $C(\bold{i}) \subset \Delta \times \Z_{\geq 0}^N$ is the set of $(\lambda, \vec{t})$ defined by the following inequalities. 
\begin{enumerate}
\item $\sum_k d_k(\pi) t_k \geq 0$ for any $\bold{i}-$trail $\omega_i \to w_0s_i\omega_i$ in $V(\omega_i),$ for all fundamental weights $\omega_j$ of the dual Langlands group.\\   
\item $t_k \leq H_{\alpha_{i_k}}(\lambda) -\sum_{\ell = k+1}^N a_{i_{\ell}, i_k} t_{\ell}$ for $k = 1, \ldots, N.$\\
\end{enumerate}
\end{definition}  

After specializing the $\Delta$ component at a dominant weight, one obtains the string parameters which index the elements of the dual canonical basis in $V(\lambda).$  For the algebra 
$R_G = \bigoplus_{\lambda \in \Delta} V(\lambda)$ the $\Delta$ parameter on the dual canonical basis corresponds to the grading by the right $T$-action.  The left $T$-action gives a grading on $R_G$ where the dual canonical basis element $b_{\lambda, \vec {s}}$ has weight $(\sum s_i\alpha_i) - \lambda.$  Elements of $B$ can then be ordered first by dominant weight, with ties broken by the lexicographic order on the string parameters.  Recall that Caldero \cite{C} has shown multiplication in $R_G$ to be lower triangular with respect to this ordering, see also \cite{AB}. 
 
\begin{equation}
b_{\lambda, \vec{s}} \times b_{\eta, \vec{t}} = b_{\lambda + \eta, \vec{s} + \vec{t}} + \sum_{\vec{\ell} < \vec{s} + \vec{t} } C_{\lambda+ \eta, \vec{\ell}}b_{\lambda + \eta, \vec{\ell}}\\ 
\end{equation}

\noindent
As a result Caldero obtains the following theorem, the associated graded algebra
is defined in the sense of Proposition \ref{pval}.

\begin{theorem}
The algebra $R_G$ has a filtration with associated graded ring isomorphic to $\C[C(\bold{i})].$
\end{theorem}

\subsection{Bases for $R(\delta_2)$ and $R(i_{L, G})$}

The dual canonical basis $B(\lambda) \subset V(\lambda)$ is known to be a
good basis, see \cite{BZ2}, \cite{Mat}, \cite{Lu}.  This means
that its restriction to certain subspaces $V_{\eta, \beta}(\lambda) \subset V(\lambda)$ 
and $V_{I, \eta}(\lambda) \subset V(\lambda)$ is still a basis.  Here all $\chi, \beta, \lambda, \eta$ are weights, and $I$ is a collection of simple roots.  Let $e_i$ be the raising operator
in the lie algebra $\mathfrak{g}$ of $G.$

\begin{equation}
 V_{\chi, \alpha}(\lambda) = \{ v | t\circ v = \chi(t)v,  e_i^{H_i(\alpha) + 1}v = 0\} \subset V(\lambda).\\
\end{equation}

\begin{equation}
V_{I, \eta}(\lambda) = \{ v | t\circ v = \eta(t)v,  e_iv = 0, i \in I\}\subset V(\lambda).\\
\end{equation}

Notice that the $T$-weight spaces of 
$V(\lambda)$ are realized as a special case $I = \emptyset$ of the second definition.  The space $V_{I, \eta}(\lambda)$ is 
by definition the $U_L$-fixed points of weight $\eta$ for $U_L \subset L$ the Levi subgroup corresponding to $I$
the chosen subset of roots, which means we have an identification. 

\begin{equation}
V_{I, \eta}(\lambda) \cong Hom_L(V(\eta), V(\lambda))\\
\end{equation}

\noindent
Here $\eta$ is viewed as a dominant weight for $L.$
When we fix $\chi = \mu - \beta$ in the first definition, for $\mu$ and $\beta$ dominant
weights there is an identification,

\begin{equation}
V_{\mu - \beta, \beta}(\lambda) \cong Hom_G(V(\mu), V(\beta) \otimes V(\gamma)).\\
\end{equation}

\noindent
The string parameterizations allow Berenstein and Zelevinsky to produce 
the following polyhedral descriptions of the elements of the dual canonical basis 
in these multiplicity spaces. 

\begin{theorem}[Berenstein, Zelevinksy, \cite{BZ1}]\label{BZ3}
The set of dual canonical basis members which span $V_{\mu - \beta, \beta}(\lambda) \subset V(\lambda)$
are indexed by the points in $\Z_{\geq 0}^N$ such that the following hold. 

\begin{enumerate}
\item $\sum_k d_k(\pi)t_k \geq 0$ for any $\bold{i}-$trail
from $\omega_j$ to $w_0 s_j\omega_j$ in $V(\omega_j),$ for all fundamental weights $\omega_j$ of the dual Langlands group.\\

\item $-\sum_k t_k \alpha_k + \lambda  + \beta = \mu$\\

\item $\sum_k d_k(\pi) t_k \geq H_{\alpha_j}(\beta)$  for any $\bold{i}-$trail
from $s_j\omega_j$ to $w_0\omega_j$ in $V(\omega_j),$  for all fundamental weights $\omega_j$ of the dual Langlands group.\\

\item $t_k + \sum_{j > k} a_{i_k, i_j} t_j \geq H_{\alpha_{i_k}}(\lambda)$\\
\end{enumerate}

\end{theorem}

\noindent
The first and last conditions say that $(\lambda, \vec{t})$ is a member
of $C(\bold{i})$ in the fiber over the weight $\lambda,$ the second condition
says that the basis members lie in the weight $\mu-\beta$ subspace of $V(\lambda),$
and the third condition says that the appropriate raising operators annihilate the basis members.

Berenstein and Zelevinsky state this for a general semi-simple group, but it extends to
a general reductive group as follows.  The weights that define a non-zero
 $Hom_G(V(\mu), V(\lambda)\otimes V(\beta))$ are of the form $\mu' + \tau_1,$
$\beta' + \tau_2$ and $\lambda' + \tau_3$ where $\tau_i$ are characters of the center
of $G$ with $\tau_1 = \tau_2 + \tau_3$ and $\mu', \beta', \lambda'$ are dominant weights
of the semisimple part of $G.$  The subspace $V_{\mu - \beta, \beta}(\lambda)$ is the
same as the subspace $V_{\mu' - \beta' + (\tau_1 - \tau_2), \beta'}(\lambda' + \tau_3) = V_{\mu' -\beta' + \tau_3, \beta'}(\lambda' + \tau_3) = V_{\mu' - \beta', \beta'}(\lambda')\otimes \C\tau_3.$  So this space inherits the subset of the dual canonical
basis of the semi-simple part of $G$ coming from $V_{\mu' - \beta', \beta'}(\lambda')$ tensored with the character $\tau_3.$ From now on we label the dual canonical basis members in this space with the triple of dominant weights $b_{\lambda, \mu, \beta, \vec{s}} \in V_{\mu -\beta, \beta}(\lambda).$

For the branching multiplicities over $L,$ any string $\bold{i}$
of the dual canonical basis gives a parametrization of the points in $V_{\eta,I}(\lambda),$
however in special cases there is a polyhedral description.  Following \cite{BZ1}, we choose $\bold{i}$ to be a concatenation of factorizations $\bold{i}_1$ of $w_0(I),$ the longest word in the parabolic subgroup of $W$ corresponding to $I,$ and $\bold{i}_2,$
a factorization of $w_0(I)^{-1}w_0.$  In this case, the projection of $C(\bold{i})$ onto the $\Z_{\geq 0}^N$ the factor splits as a product of string cones, see Theorem 3.11 of \cite{BZ1}.

\begin{theorem}[Berenstein, Zelevinksy, \cite{BZ1}]\label{BZL}
The set of dual canonical basis members which span $V_{\eta}(\lambda) \subset V(\lambda)$
are $\bold{i}_1 \circ \bold{i}_2 = \bold{i}$- parametrized by the points in $\Z^N$ with the first $N- \ell(w_0(I)^{-1}w_0)$ coordinates equal to zero, such that the following hold. 

\begin{enumerate}
\item $\sum_{k >N- \ell(w_0(I)^{-1}w_0)} d_k(\pi)t_k \geq 0$ for any $\bold{i}_2-$trail
from $w_0(I)\omega_j$ to $w_0 s_j\omega_j$ in $V(\omega_j),$  for all fundamental weights $\omega_j$ of the dual Langlands group.\\

\item $\sum_{k >N- \ell(w_0(I)^{-1}w_0)} t_k \alpha_k = \lambda - \eta$\\

\item $t_k + \sum_{j > k} a_{i_k, i_j} t_j \geq H_{\alpha_{i_k}}(\lambda)$\\
\end{enumerate}

\end{theorem}

We label the basis members in this subspace $b_{\eta, \lambda, \vec{t}}.$
The purpose of the rest of this section is to define the two cones $C_3(\bold{i})$ and $C_L(\bold{i})$ 
which have cross-sections equal to the above polytopes and index bases $B_3(\bold{i})$
and $B_L(\bold{i})$ in the two branching algebras $R(\delta_2)$ and $R(i_{L,G}).$ In order to do so, we will relate both algebras to $R_G.$ Let $U_-$ and $U_+$ be the maximal 
unipotent subgroups of $G$ associated to the negative and positive simple roots.
Following Zhelobenko, Chapter XVIII, page 383 of \cite{Zh}, we have the following commutative diagram.

$$
\begin{CD}
T\times G/U_+ @>>> G/U_+ \times G/U_+ @>>> U_{-} \ql [G/U_+ \times G/U_+]\\
@AAA @AAA @AAA \\
T \times U_{-}\times T @>>> U_{-} \times T \times U_{-} \times T @>>> U_{-} \ql [U_{-} \times T \times U_{-} \times T]\\
\end{CD}
$$

\noindent
The bottom row of this diagram is dense in the top row. An invariant function $f \in \C[G/U_+ \times G/U_+]^{U_{-}}$ satisfies

\begin{equation}
f(u_1t_1, u_2t_2) = f(t_1, u_1^{-1}u_2t_2),\\
\end{equation}

\noindent
so we can realize $\C[G/U_+ \times G/U_+]^{U_{-}}$ as a subalgebra of $R(G)\otimes \C[T].$ 
The algebra $\C[G/U_+ \times G/U_+]^{U_{-}}$ is isomorphic to $\C[G/U_+ \times G/U_+ \times G/U_+]^G  \cong  R(\delta_2).$
From \cite{Zh}, we get that the map on the graded pieces is
\begin{equation}
Hom_{U_{-}}(\C v_{\mu}, V(\lambda)\otimes V(\beta)) \to V_{\mu - \beta ,\beta}(\lambda) \otimes \C v_{\beta} \subset V(\lambda)\otimes \C v_{\beta}.\\
\end{equation}

In this way, $R(\delta_2)$ inherits
the dual canonical basis $B_3 = \coprod_{\mu, \beta, \lambda} B(\lambda)\otimes v_{\beta} \cap  V_{\mu - \beta ,\beta}(\lambda) \otimes v_{\beta},$ 
along with the multiplication operation in $R_G\otimes \C[T].$  This implies that the basis in $R(\delta_2)$ has the lower-triangularity property with respect to multiplication in $R(\delta_2).$

The algebra $R_G \otimes \C[T]$ has $4$ torus actions by $T \subset G,$ namely the left and right action on both components. 
We label these $T_1 \times T_2 \times T_3 \times T_4,$ from left to right.  A graded component of this action is the space
$V_{\mu}(\lambda) \otimes v_{\beta}.$  Here $\mu$ is the character of $T_1$, and $\lambda$ is the character
of $T_2.$  The character $\beta$ and $-\beta$ are associated to $T_3$ and $T_4.$

The algebra $R(\delta_2)$ has $3$ distinguished torus actions, one for each representation parameter, the character spaces of $T_a \times T_b \times T_c$ are the spaces $Hom_G(V(\mu), V(\lambda ) \otimes V(\beta)),$ where $T_a$ has character $\mu,$
$T_b$ has character $\lambda$ and $T_c$ has character $\beta.$  Under the map above, 
$T_a$ is the diagonal in $T_1 \times T_3,$ $T_b = T_2$ and $T_c = T_4.$
In this way, the subspace  $V_{\mu - \beta ,\beta}(\lambda)$ has $T_1$-character equal to $\mu-\beta$, $T_2$-character equal to $\lambda$ and $T_3 \times T_4$ character $(-\beta, \beta),$ so this space corresponds to the $(\mu - \beta + \beta, \lambda, \beta) = (\mu, \lambda, \beta)$ character space of $T_a \times T_b \times T_c.$

\begin{definition}
For a string parameterization $\bold{i},$ the cone $C_3(\bold{i})$ is defined by the following inequalities
on $(\lambda, \vec{t} ,\beta) \in \Delta \times C(\bold{i}) \times \Delta \subset \Delta \times \Z_{\geq 0}^N \times \Delta.$

\begin{enumerate}
\item $\sum_k d_k(\pi)t_k \geq 0$ for any $\bold{i}-$trail
from $\omega_j$ to $w_0 s_j\omega_j$ in $V(\omega_j),$  for all fundamental weights $\omega_j$ of the dual Langlands group.\\

\item $-\sum_k t_k \alpha_k + \lambda  + \beta \in \Delta$\\

\item $\sum_k d_k(\pi) t_k \geq H_{\alpha_j}(\beta)$  for any $\bold{i}-$trail
from $s_j\omega_j$ to $w_0\omega_j$ in $V(\omega_j),$   for all fundamental weights $\omega_j$ of the dual Langlands group.\\

\item $t_k + \sum_{j > k} a_{i_k, i_j} t_j \geq H_{\alpha_{i_j}}(\lambda)$\\
\end{enumerate}

\end{definition}

We can prove Theorem \ref{basisinherit} as follows. If $b_{\lambda, \vec{t}}\otimes v_{\beta} \in R_G \otimes \C[T]$ satisfies the above
conditions, then we can recover a third weight $\mu = \sum_k t_k \alpha_k - \lambda + \beta,$ and by the theorem of Berenstein and Zelevinksy, $b_{\lambda, \vec{t}} \otimes v_{\beta} \in V_{\mu - \beta, \beta}(\lambda)\otimes \C v_{\beta} \subset R(\delta_2) \subset R_G\otimes \C[T].$  By construction, if $b_{\lambda, \vec{s}}\otimes v_{\beta}$ is in $R(\delta_2),$ then it must be in some $V_{\mu - \beta, \beta}(\lambda)\otimes \C v_{\beta},$ and must have string parameters in $C_3(\bold{i}).$  

 Note that $C_3(\bold{i})$ is a rational cone, and specializing the parameters 
$\lambda, \beta, \mu$ yields the conditions from \cite{BZ1} which index the members
of the dual canonical basis in the space $V_{\mu - \beta ,\beta}(\lambda).$  Just as the elements of $C(\bold{i})$ have an ordering, we place an order on $C_3(\bold{i})$ where $(\lambda_1, \lambda_2, \lambda_3, \vec{s}) < (\eta_1, \eta_2, \eta_3, \vec{t})$
if $(\lambda_1, \lambda_2, \lambda_3) < (\eta_1, \eta_2, \eta_3)$ as dominant weights of $G^3,$
or $(\lambda_1, \lambda_2, \lambda_3) = (\eta_1, \eta_2, \eta_3)$ and $\vec{s} < \vec{t}$ 
lexicographically.  Multiplication in $R(\delta_2)$ is lower triangular because the same holds 
for $R_G\otimes \C[T].$

We can now carry out the same construction for the branching algebra $R(i_{L,G}).$
First we recall the two ways to see a branching algebra.

\begin{equation}
R(i_{L,G}) = {}^L[R_L\otimes R_G] \cong R_G^{U_L}\\
\end{equation}

\noindent
The algebra $R_G^{U_L}$ sits inside $R_G$ as the direct
sum of the invariant spaces $V(\lambda)^{U_L}.$  Each of
these spaces in turn decomposes over the dominant weights of $L.$

\begin{equation}
V(\lambda)^{U_L} = \bigoplus_{\eta \in \Delta_L} V_{I, \eta}(\lambda)\\
\end{equation}

\noindent
We have $V_{I, \eta}(\lambda) \cong Hom_L(V(\eta), V(\lambda)).$
The algebra $R(i_{L,G})$ comes with a natural bigrading by $\Delta_L \times \Delta_G,$
which matches the bigrading by pairs of dominant weights on the spaces $V_{I,\eta}(\lambda).$
By Theorem \ref{BZL} above, each space $V_{I, \eta}(\lambda)$ has a basis obtained by 
$B(I, \eta, \lambda) = B\cap V_{I, \eta}(\lambda),$ note that this is independent of the
string parameterization.  We choose a $\bold{i}$ that is adapted to $L$ as above, 
with the first part of the string a reduced factorization $w_0(I)^{-1}w_0$ and the second
part a factorization of $w_0(I).$

\begin{definition}
We define the cone $C_L(\bold{i}) \subset \Delta \times \Z_{\geq 0}^N$ for $\bold{i} = \bold{i}_1 \circ \bold{i}_2$ to 
be the subcone of $C(\bold{i})$ of strings with first $N- \ell(w_0(I)^{-1}w_0)$ entries $0,$ which satisfy the following inequalities. 

\begin{enumerate}
\item $\sum_{k >N- \ell(w_0(I)^{-1}w_0)} d_k(\pi)t_k \geq 0$ for any $\bold{i}_2-$trail
from $w_0(I)\omega_j$ to $w_0 s_j\omega_j$ in $V(\omega_j),$  for all fundamental weights $\omega_j$ of the dual Langlands group.\\

\item $\sum_{k >N- \ell(w_0(I)^{-1}w_0)} t_k \alpha_k -\lambda \in \Delta_L$\\

\item $t_k + \sum_{j > k} a_{i_k, i_j} t_j \geq H_{\alpha_{i_k}}(\lambda)$\\
\end{enumerate}

\end{definition}

 The cone $C_L(\bold{i})$ also has a partial ordering given by $(\eta_1, \lambda_1, \vec{s}) < (\eta_2, \lambda_2, \vec{t})$ if $(\eta_1, \lambda_1) < (\eta_2, \lambda_2)$
as dominant weights of $L \times G,$ or $(\eta_1, \lambda_1) = (\eta_2, \lambda_2)$ and $\vec{s} < \vec{t}$ 
lexicographically.  Both $R(\delta_2)$ and $R(i_{L,G})$ are then filtered by partially ordered monoids, as in
Proposition \ref{pval}.  The weight spaces in this filtration are all $1$ dimensional, and the associated graded rings
are $\C[C_3(\bold{i})]$ and $\C[C_L(\bold{i})]$ respectively.   We can apply a tuple of coweights to the dominant weight components above to obtain a valuation into $\Z^{N+1}$ with the lexicographic ordering by Proposition \ref{pval}.  Since the basis is indexed by the points of a rational cone, we can also use, Proposition \ref{pval2}, and Proposition \ref{reese} to prove Theorem \ref{basisinherit}.

\section{Degenerations of $R(\delta_n)$ and $R(i_{L,G})$}

In this section we combine valuations constructed from factorization diagrams with
those defined by the dual canonical basis to obtain toric degenerations of $R(\delta_n)$
and $R(i_{L,G}).$  We go through the proof for $\delta_n$ only, as the steps are identical for $i_{L,G}.$

After choosing a tree $\tree$ to give a filtration of $R(\delta_n),$  we assign a string $\bold{i}(v)$ to each internal 
vertex $v \in V(\tree),$ we call this a $\tree$-string.  Each space $W(\tree, \vec{\lambda})$ is a tensor product of spaces
$Hom_G(V(\lambda_1(v)), V(\lambda_2(v)) \otimes V(\lambda_3(v))),$ each with a partially ordered basis of dual canonical
basis members $b_{\lambda_1(v), \lambda_2(v),\lambda_3(v), \vec{s}}.$  The space $W(\tree, \vec{\lambda})$ therefore
has a basis of factorization diagrams labeled by string parameters at each internal vertex. 

\begin{equation}
b_{\tree, \vec{\lambda}, \vec{s}} = \bigotimes_{v \in V(\tree)} b_{\lambda_2(v), \lambda_1(v),\lambda_3(v), \vec{s}(v)}\\
\end{equation}

\noindent
The labels form a cone $C_{\tree}(\bold{i}),$ which is a fiber product of $|V(\tree)|-$ copies of $C_3(\bold{i})$ over copies of $\Delta_G$ assigned to the edges of $\tree.$  This cone has a partial ordering, where $(\tree, \vec{\lambda}, \vec{s}) < (\tree, \vec{\eta}, \vec{t})$ if $\vec{\lambda} < \vec{\eta}$ as dominant weights of $G^{|E(\tree)|},$ or $\vec{\lambda} = \vec{\eta},$ and $\vec{s} < \vec{t}$ lexicographically, by the ordering induced by choosing an ordering on the internal vertices of $\tree.$  

\begin{proposition}
Multiplication in $R(\delta_n)$ with respect to the basis $b_{\tree, \vec{\gamma}, \vec{s}}$ is lower triangular. 
\begin{equation} 
b_{\tree, \vec{\eta}, \vec{t}} \times b_{\tree, \vec{\gamma}, \vec{s}} = b_{\tree, \vec{\eta} + \vec{\gamma}, \vec{t} + \vec{s}} + \sum C_{\tree, \vec{\beta}, \vec{r}} b_{\tree, \vec{\beta}, \vec{r}}\\
\end{equation}

where $(\tree, \vec{\beta}, \vec{r}) <  (\tree, \vec{\eta} + \vec{\gamma}, \vec{t} + \vec{s}).$
\end{proposition}

\begin{proof}
This follows from the lower triangularity of the multiplication rules for branching filtrations and the dual canonical bases.   
\end{proof}
\noindent
We can now apply Propositions \ref{pval2} and \ref{reese} as we did on the algebras $R(\delta_2)$ and $R(i_{L,G})$ to obtain Theorems \ref{diagtoric} and \ref{levitoric}.

\section{Tensor products in type $A$}\label{examples}

  The cones
$C_3(\bold{i})$ are useful in that they give toric degenerations
of the full tensor algebra, but they are a challenge to write down in practice. 
For type $A$ partial results were obtained by Berenstein and Zelevinksy in \cite{BZ2},
with a full answer given by Gleizer and Postnikov in \cite{GP} in terms of a device they call
"web functions." These are equivalent to the honeycomb polytopes
of Knutson - Tao used in the proof of the saturation conjecture for $GL_n(\C),$ \cite{KT}, and the Berenstein Zelevinsky triangles, first described in \cite{BZ2}.
These are assignments of non-negative integers to the intersection points of diagrams
like figure \ref{triangle} below, with the condition that pairs on opposite sides of the same hexagon add to the same number.

\begin{figure}[htbp]
\centering
\includegraphics[scale = 0.29]{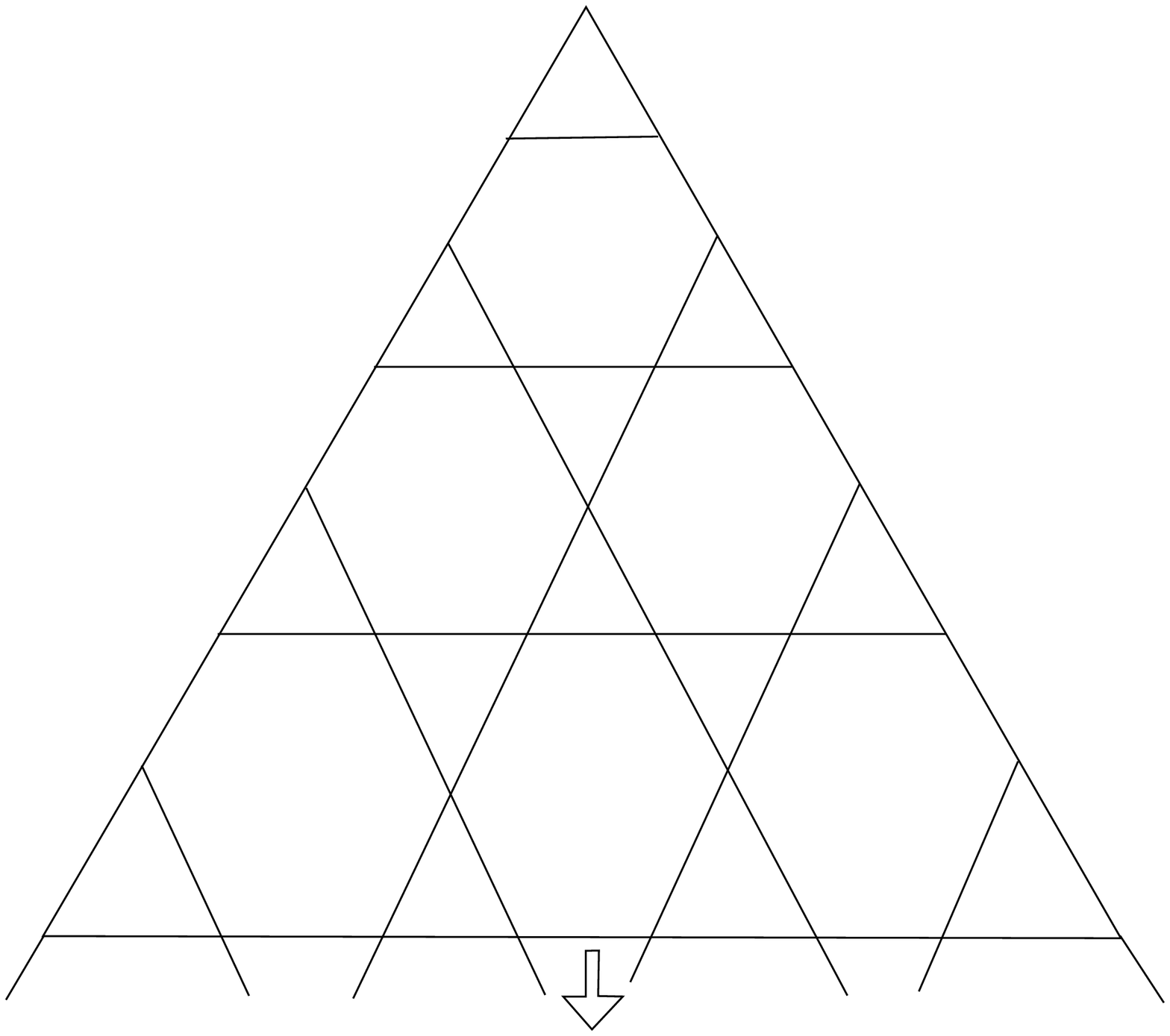}
\caption{}
\label{triangle}
\end{figure}

Note that the sides of this triangle are lined with small triangles, $K_3(SL_m(\C))$ is all such assignments to the diagram with $m-1$ small triangles to a side. We orient the triangle counter-clockwise. For a particular weighting $w$ of this diagram, 
let $\lambda_1(w),$ $\lambda_2(w),$ and $\lambda_3(w)$
be the vectors of numbers obtained from sides $1$ $2$ and $3$ of the diagram
by adding the pairs of numbers assigned to the vertices of the little triangles, so the $j-$th
entry of $\lambda_i(w)$ is the sum of the two numbers assigned to the two vertices on the $j-th$ little triangle bordering the $i$-th side. These numbers define dominant weights of $SL_m(\C)$ by $\sum \lambda_i(w)_j \omega_j,$ where $\omega_j$ is the $j-$th fundamental weight.   A weighting $w \in K_3(SL_m(\C))$ represents an invariant in the triple tensor product $V(\lambda_1(w))\otimes V(\lambda_2(w))\otimes V(\lambda_3(w)).$
The following is essentially proved in \cite{BZ3} and \cite{BZ2}.

\begin{proposition}
For a particular choice of $\bold{i},$ there is a commutative diagram of linear maps of cones.

$$
\begin{CD}
K_3 @>>> C_3(\bold{i})\\
@VVV @VVV\\
\R^{3r} @>d_2>> \R^{3r}\\
\end{CD}
$$

\noindent
Here the vertical maps are projection onto triples of dominant weights, the bottom horizontal map is duality on the middle dominant weight, and the top map is a lattice isomorphism of cones.  
\end{proposition}

The string cone $C_3(\bold{i})$ in this theorem is the cone of "partitions,"  defined in \cite{BZ2}, which realize the Littlewood-Richardson rule.  The full tensor product algebra $R(\delta_n)$ therefor degenerates to the toric algebra of the fiber product 
of $K_3(SL_m(\C))$ over a trivalent tree $\tree.$  We represent elements of this monoid with labellings of "quilt diagrams" of BZ triangles, shown in figure \ref{quilt}.

\begin{figure}[htbp]
\centering
\includegraphics[scale = 0.45]{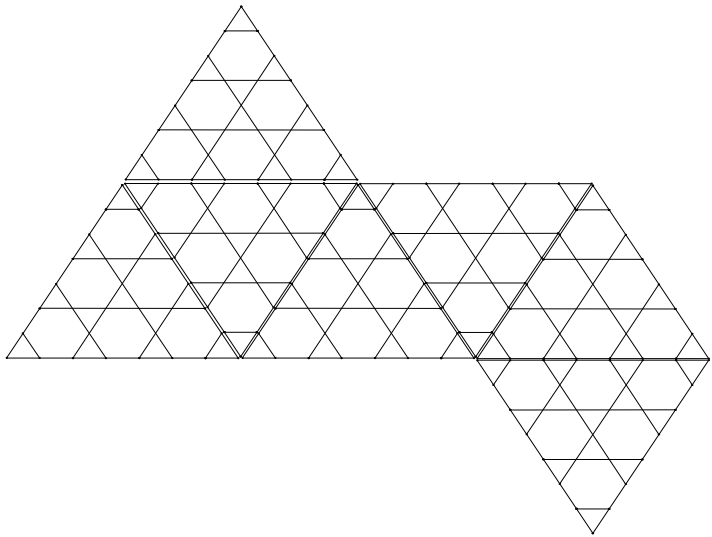}
\caption{A quilt diagram for $R(\delta_8)$}
\label{quilt}
\end{figure}

Each triangle is a BZ triangle, triangles which border each other must have the 
same dominant weights along their edges, and $\tree$ can be recovered as the dual tree
to the arrangement of triangles.  We refer to this semigroup as $K_{\tree}(SL_m(\C)).$

We give $K_{\tree}(SL_m(\C))$ as a particular product of Theorem \ref{diagtoric} which appears to be amenable to computational methods.  Deciphering the structure of this semigroup for small $m$ should not be difficult.   We also present BZ quilts as a possible tool for studying other interesting algebras.  The subsemigroup $K_{\tree}(\vec{r}, SL_m(\C))$ of quilts which have a multiple of $(r_1 \omega_1, \ldots, r_n \omega_1)$ for boundary weights is a toric degeneration of the projective coordinate ring $\C[Gr_m(\C^n)//_{\vec{r}} T]$ the weight variety at $\vec{r}$ of the Grassmannian.   This is a classical algebra from invariant theory, and not much is known about it outside the case $m = 2,$ see \cite{HMSV}.   We also observe that a similar analysis could be carried out using a Levi branching degeneration using Theorem \ref{levitoric}.


\end{document}